\documentclass[11pt,twoside, letter]{article}

\usepackage{amsthm, amsmath, amssymb, amsfonts, graphicx, epsfig}
\usepackage{accents}
\usepackage{kbordermatrix}

\usepackage{bbm}
\usepackage{mathrsfs}

\usepackage{float}
\restylefloat{table}

\usepackage{enumerate}

\usepackage[usenames,dvipsnames]{color}

\usepackage[colorlinks=true, pdfstartview=FitV, linkcolor=blue,
            citecolor=blue, urlcolor=blue]{hyperref}
\usepackage[usenames]{color}

\usepackage{url}
\makeatletter\def\url@leostyle{%
 \@ifundefined{selectfont}{\def\UrlFont{\sf}}{\def\UrlFont{\scriptsize\ttfamily}}} \makeatother

\usepackage[framemethod=default]{mdframed}

\usepackage{array}

\usepackage[margin=1.0in, letterpaper]{geometry}

\newtheorem{theorem}{Theorem}
\newtheorem{proposition}[theorem]{Proposition}
\newtheorem{lemma}[theorem]{Lemma}
\newtheorem{corollary}[theorem]{Corollary}

\theoremstyle{definition}

\theoremstyle{remark}
\newtheorem{remark}[theorem]{Remark}

\numberwithin{equation}{section}
\numberwithin{theorem}{section}

\definecolor{Red}{rgb}{1.0,0,0.0}
\definecolor{Blue}{rgb}{0,0.0,1.0}

\def\cC{\mathcal{C}}
\def\cD{\mathcal{D}}

\def\cF{\mathcal{F}}

\def\cL{\mathcal{L}}

\def\cP{\mathcal{P}}
\def\cQ{\mathcal{Q}}

\def\bC{\mathbb{C}}

\def\bE{\mathbb{E}}
\def\bF{\mathbb{F}}

\def\bN{\mathbb{N}}

\def\bP{\mathbb{P}}

\def\bR{\mathbb{R}}

\def\sF{\mathscr{F}}

\def\mA{\mathsf{A}}
\def\mB{\mathsf{B}}
\def\mC{\mathsf{C}}
\def\mD{\mathsf{D}}

\def\mG{\mathsf{G}}
\def\mH{\mathsf{H}}
\def\mI{\mathsf{I}}

\def\mV{\mathsf{V}}

\newcommand{\wt}{\widetilde}
\newcommand{\wh}{\widehat}

\newcommand{\1}{\mathbbm{1}}            
\newcommand{\set}[1]{\{#1\}}            
\renewcommand{\mid}{\;|\;}              

\DeclareMathOperator{\dif}{d \!}        

\title{Wiener-Hopf factorization for time-inhomogeneous Markov chains and its application}

\author{
        Tomasz R. Bielecki\,\thanks{Department of Applied Mathematics, Illinois Institute of Technology
       \newline \hspace*{1.45em}  10 W 32nd Str, Building E1, Room 208, Chicago, IL 60616, USA
       \newline \hspace*{1.45em}  Emails: \url{tbielecki@iit.edu} (T. R. Bielecki), \url{cialenco@iit.edu} (I. Cialenco), \url{rgong2@iit.edu} (R. Gong) and \url{yhuang37@hawk.iit.edu} (Y. Huang)
       \newline \hspace*{1.45em}  URLs: \url{http://math.iit.edu/\~bielecki} (T. R. Bielecki),  \url{http://math.iit.edu/\~igor} (I. Cialenco) and \url{http://mypages.iit.edu/\~rgong2} (R. Gong)
        \vspace{0.5em}} ,
\and
        Igor Cialenco\,\footnotemark[1] ,
\and
        Ruoting Gong\,\footnotemark[1] ,
         \vspace{0.5em}
\and
        Yicong Huang\,\footnotemark[1]
         \vspace{0.5em}}

\date{ {\small
First Circulated: January 14, 2018
}}

\begin{document}

\maketitle

\vspace{-2em}



\smallskip

{\footnotesize
\begin{tabular}{l@{} p{350pt}}
  \hline \\[-.2em]
  \textsc{Abstract}: \ & In this paper we derive the Wiener-Hopf factorization  for a finite-state time-inhomogeneous Markov chain. To the best of our knowledge, this study is the first attempt to investigate the Wiener-Hopf factorization for time-inhomogeneous Markov chains. In this work we only deal with a special class of time-inhomogeneous Markovian generators, namely piece-wise constant, which allows to use an appropriately tailored randomization technique.  Besides the mathematical importance of the Wiener-Hopf factorization methodology, there is also an important computational aspect: it allows for efficient computation of important functionals of Markov chains. \\[0.5em]
\textsc{Keywords:} \ & Wiener-Hopf factorization, inhomogeneous Markov chain, fluctuation theory, randomization method, additive functional.  \\
\textsc{MSC2010:} \ & 60J27, 60J28, 60K25 \\[1em]
  \hline
\end{tabular}
}



\section{Introduction}
In this paper we derive the Wiener-Hopf factorization (WHf) for a finite-state time-inhomogeneous Markov chain. As far as we know, our study is the first attempt to investigate the Wiener-Hopf factorization for time-inhomogeneous Markov chains. In this pioneering study we only deal with a special class of time-inhomogeneous Markovian generators, namely piece-wise constant. This allows to use an appropriately tailored randomization technique.  Besides the mathematical importance of the WHf, there is an important computational aspect: this methodology allows for very efficient computation of important functionals of Markov chains.

The Wiener-Hopf factorization  for finite-state  Markov chains  was  originally derived in \cite{Barlow1980} in the time-homogeneous case; see also \cite{London1982} and  \cite{Williams1991}. For the WHf in case of time-homogeneous Feller Markov processes we refer to \cite{Williams2008}. For some related applied work we refer to \cite{Avram2003}, which deals with the ruin problem, and to \cite{Asmussen1995,rogers1994fluid,rogers_shi_1994} that study fluid models.
In addition, \cite{Kennedy1990} studies the so called ``noisy'' Wiener-Hopf factorizations; for applications see \cite{Asmussen1995,rogers1994fluid,rogers_shi_1994,Jobert2006,Jiang2008,Mijatovic2011,Jiang2012,Hieber2014,Hainaut2016}.

It needs to be stressed that even though the classical WHf of \cite{Barlow1980} can be applied to the generator matrix, say $\mG_t$, of a time-inhomogeneous Markov chain $X$ at every time $t$, these factorizations do not have any probabilistic meaning with regard to the process $X$. In particular, they are of no use for computing functionals such as \eqref{eq:tau0+}-\eqref{eq:taut-} below. So, a relevant WHf for a time-inhomogeneous Markov chain requires a different approach than the one that would just directly apply the results of \cite{Barlow1980} to each $\mG_t$, $t\geq 0$.

The paper is organized as follows. In Section \ref{sec:setup} we provide a motivation and the setup up for our problem. In Section \ref{subsec:th} we introduce a randomization method and we give the main results of the paper. Section \ref{sec:Numerical} provides a numerical algorithm for computing our version of the WHf and its application in a specific example. Finally, we give some supporting results in the Appendix.

\section{Motivation and problem set-up}\label{sec:setup}

Let $\mathbf{E}$ be a finite set, $(\Omega,\sF,\bP)$ be a complete probability space, and $X:=(X_{t})_{t\geq 0}$ be a {\it time-inhomogeneous} Markov chain on $(\Omega,\sF,\bP)$ with state space $\mathbf{E}$ and generator function $\mG=\{\mG_{t},t\geq 0\}$. In particular, each $\mG_{t}$ is a $|\mathbf{E}|\times |\mathbf{E}|$ matrix. We assume that $\bP\left(X_{0}=i\right)>0$ for each $i\in\mathbf{E}$  and we let $\bP^{i}$ be the probability measure on $(\Omega,\sF)$ defined by
\begin{align*}
\bP^{i}(A):=\bP\left(A\,|\,X_{0}=i\right),\quad A\in\sF,
\end{align*}
with $\bE^{i}$ denoting the associated expectation.

In this paper we assume that the generator $\mG$ is piecewise constant, namely we assume that
\begin{align*}
\mG_{t}=\left\{\begin{array}{ll} \mG_{1},\quad &\text{if }\,s_{0}\leq t<s_{1},\\ \mG_{2},\quad &\text{if }\,s_{1}\leq t<s_{2},\\ \,\,\vdots & \\ \mG_{n},\quad &\text{if }\,s_{n-1}\leq t<s_{n},\\ \mG_{n+1},\quad &\text{if }\,t\geq s_{n}, \end{array}\right.
\end{align*}
 for some $n\in\bN$ and $0=s_{0}<s_{1}<\ldots<s_{n}$.
 Without loss of generality we assume that $\mG_{1},\ldots,\mG_{n+1}$ are not sub-Markovian. That is, the sums of row elements of $\mG_{k}$ are all zero, for any $k=1,\ldots,n+1$. The results of this paper carry over to the sub-Markovian case by the standard augmentation of the state space.

Next, we consider a function $v:\mathbf{E}\rightarrow\bR\setminus\{0\}$ and we put
\begin{align*}
\mathbf{E}^{+}:=\left\{\left.i\in\mathbf{E}\,\right|v(i)>0\right\}\quad\text{and}\quad\mathbf{E}^{-}:=\left\{\left.i\in\mathbf{E}\,\right|v(i)<0\right\}.
\end{align*}
We  also define the additive functional
\begin{align*}
\varphi_{t}:=\int_{0}^{t}v(X_{u})\dif u,\quad t\geq 0,
\end{align*}
and the first passage times
\begin{align*}
\tau_{t}^{+}:=\inf\left\{\left.r\geq 0\,\right|\varphi_{r}>t\right\}\quad\text{and}\quad\tau_{t}^{-}:=\inf\left\{\left.r\geq 0\,\right|\varphi_{r}<-t\right\}.
\end{align*}

The main goal of this paper is to apply the Wiener-Hopf factorization technique, which we work out in Section \ref{subsec:th},  to compute the following expectations,
\begin{align}\label{eq:tau0+}
\Pi_{c}^{+}(i,j;s_{1},\ldots,s_{n})&:=\mathbb{E}\left(e^{-c\tau_{0}^{+}}\1_{\{X_{\tau_{0}^{+}}=j\}}|X_{0}=i\right),\quad i\in\mathbf{E}^{-},\,j\in\mathbf{E}^{+},\\
\label{eq:taut+}
\Psi_{c}^{+}(\ell,i,j;s_{1},\ldots,s_{n})&:=\mathbb{E}\left(e^{-c\tau_{\ell}^{+}}\1_{\{X_{\tau_{\ell}^{+}}=j\}}|X_{0}=i\right),\quad i\in\mathbf{E}^{+},\,j\in\mathbf{E}^{+},\,\ell>0, \\
\label{eq:tau0-}
\Pi_{c}^{-}(i,j;s_{1},\ldots,s_{n})&:=\mathbb{E}\left(e^{-c\tau_{0}^{-}}\1_{\{X_{\tau_{0}^{-}}=j\}}|X_{0}=i\right),\quad i\in\mathbf{E}^{+},\,j\in\mathbf{E}^{-},\\
\label{eq:taut-}
\Psi_{c}^{-}(\ell,i,j;s_{1},\ldots,s_{n})&:=\mathbb{E}\left(e^{-c\tau_{\ell}^{-}}\1_{\{X_{\tau_{\ell}^{-}}=j\}}|X_{0}=i\right),\quad i\in\mathbf{E}^{-},\,j\in\mathbf{E}^{-},\,\ell>0.
\end{align}
We will focus on the computation of $\Pi_{c}^{+}(i,j;s_{1},\ldots,s_{n})$ and $\Psi_{c}^{+}(\ell,i,j;s_{1},\ldots,s_{n})$. By symmetry, analogous results  can be obtained for $\Pi_{c}^{-}(i,j;s_{1},\ldots,s_{n})$ and $\Psi_{c}^{-}(\ell,i,j;s_{1},\ldots,s_{n})$.
To simplify the notations, we will frequently write $\Pi_{c}^{+}(i,j)$ and $\Psi_{c}^{+}(\ell,i,j)$ in place of $\Pi_{c}^{+}(i,j;s_{1},\ldots,s_{n})$ and  $\Psi_{c}^{+}(\ell,i,j;s_{1},\ldots,s_{n})$, respectively.

\section{A randomization method and the Wiener-Hopf factorization}\label{subsec:th}

In this section we construct a {\it time-homogeneous} Markov chain associated to $X$, by randomizing the discontinuity times $s_{1},\ldots,s_{n}$ of the generator $\mG$. This key construction will allow us to compute the expectations \eqref{eq:tau0+} and \eqref{eq:taut+} using analogous expectations corresponding to this time-homogeneous chain. The latter expectations can be computed using  Wiener-Hopf factorization theory of \cite{Barlow1980}.

Define $\bN_{n}:={\{0,\ldots,n\}}$, $\wt{\mathbf{E}}:=\bN_{n}\times\mathbf{E}$ and let $(\wt{\Omega},\wt{\cF},\wt{\bP})$ be a complete probability space. Next, let us consider a {\it time-homogeneous} Markov chain, say  $Z=(N,Y):=(N_{t},Y_{t})_{t\geq 0}$, defined on $(\wt{\Omega},\wt{\cF},\wt{\bP})$, taking values in $\wt{\mathbf{E}}$ and with generator matrix $\wt{\mG}\left((n_{1},j_{1}),(n_{2},j_{2}) \right)_{(n_{1},j_{1}),(n_{2},j_{2})\in \wt{\mathbf{E}}}$ given as
\begin{align*}
\wt{\mG}=\kbordermatrix{ & \{0\}\times\mathbf{E} & \{1\}\times\mathbf{E} & \cdots & \{n-1\}\times\mathbf{E} & \{n\}\times\mathbf{E} \\ \{0\}\times\mathbf{E} & \mG_{1}-q_{1}\mI & q_{1}\mI & \cdots & 0 & 0 \\ \{1\}\times\mathbf{E} & 0 & \mG_{2}-q_{2}\mI & \cdots & 0 & 0 \\ \vdots & \vdots & \vdots & \ddots & \vdots & \vdots \\ \{n-1\}\times\mathbf{E} & 0 & 0 & \cdots & \mG_{n}-q_{n}\mI & q_{n}\mI \\ \{n\}\times\mathbf{E} & 0 & 0 & \cdots & 0 & \mG_{n+1}},
\end{align*}
where $q_{1},\ldots,q_{n}$ are positive constants and $\mI$ is the identity matrix. For each $i\in\mathbf{E}$, we define  the probability measure $\wt{\bP}^{i}$  on $(\wt{\Omega},\wt{\cF})$ by
\begin{align}\label{eq:TildePi}
\wt{\bP}^{i}(A):=\wt{\bP}\left(\left.A\,\right|Z_{0}=(0,i)\right),\quad A\in\wt{\cF}.
\end{align}
The next  result regards the Markov property of process $N$.

\begin{proposition}\label{prop:NMC}
For any $i\in\mathbf{E}$, the process $N$ is a time-homogeneous Markov chain under $\wt{\bP}^{i}$, with generator matrix given by
\begin{align*}
\wt{\mG}_{N}=\kbordermatrix{ & 0 & 1 & \cdots & n-1 & n \\ 0 & -q_{1} & q_{1} & \cdots & 0 & 0 \\ 1 & 0 & -q_{2} & \cdots & 0 & 0 \\ \vdots & \vdots & \vdots & \ddots & \vdots & \vdots \\ n-1 & 0 & 0 & \cdots & -q_{n} & q_{n} \\ n & 0 & 0 & \cdots & 0 & 0}.
\end{align*}
\end{proposition}

\begin{proof} We will proceed in three steps.

\smallskip
\noindent
\textsf{Step 1.} We start by showing that
\begin{align}\label{eq:TildeGtoGN}
\sum_{j_{2}\in\mathbf{E}}\left(\wt{\mG}^{k}\right)\left((n_{1},j_{1}),(n_{2},j_{2}) \right)=\left(\wt{\mG}_{N}^{k}\right)(n_{1},n_{2}),
\end{align}
for any $j_{1}\in\mathbf{E}$, $k\in\mathbb{N}$, and $0\leq n_{1},n_{2}\leq n$,
In particular, note that the left-hand-side of \eqref{eq:TildeGtoGN} does not depend on $j_{1}$.

We will prove \eqref{eq:TildeGtoGN} by induction in $k$. Clearly \eqref{eq:TildeGtoGN} holds true for $k=1$.
Next, assume that \eqref{eq:TildeGtoGN} holds for some $k=\ell\in\mathbb{N}$. Now, for $\ell+1$,
\begin{align*}
\sum_{j_{2}\in\mathbf{E}}\left(\wt{\mG}^{\ell+1}\right)\left((n_{1},j_{1}),(n_{2},j_{2})\right)&=\sum_{j_{2}\in\mathbf{E}}\sum_{m=0}^{n}\sum_{j\in\mathbf{E}}\left(\wt{\mG}^{\ell}\right)\left((n_{1},j_{1}),(m,j)\right)\wt{\mG}\left((m,j),(n_{2},j_{2})\right)\\
&=\sum_{m=0}^{n}\sum_{j\in\mathbf{E}}\left(\wt{\mG}^{\ell}\right)\left((n_{1},j_{1}),(m,j)\right)\sum_{j_{2}\in\mathbf{E}}\wt{\mG}\left((m,j),(n_{2},j_{2})\right)\\
&=\sum_{m=0}^{n}\sum_{j\in\mathbf{E}}\left(\wt{\mG}^{\ell}\right)\left((n_{1},j_{1}),(m,j)\right)\wt{\mG}_{N}(m,n_{2})\\
&=\sum_{m=0}^{n}\left(\wt{\mG}_{N}^{\ell}\right)(n_{1},m)\wt{\mG}_{N}(m,n_{2})=\left(\wt{\mG}_{N}^{\ell+1}\right)(n_{1},n_{2}),
\end{align*}
where we used the inductive assumptions for $k=1$ and $k=\ell$ in the third and the fourth equalities, respectively. Hence \eqref{eq:TildeGtoGN} is established.

\smallskip\noindent
\textsf{Step 2.} We will show that
\begin{align}\label{eq:NNotDependonY}
\wt{\bP}^{i}\left(\left.N_{t+s}=n_{2}\,\right|N_{t}=n_{1}\right)=\wt{\bP}^{i}\!
\left(\left.N_{t+s}=n_{2}\,\right|N_{t}=n_{1},Y_{t}=j\right)=e^{s\wt{\mG}_{N}}(n_{1},n_{2}),
\end{align}
for any $t,s\geq 0$, $j\in\mathbf{E}$, and $0\leq n_{1}\leq n_{2}\leq n$.
In particular, note that the left-hand side of \eqref{eq:NNotDependonY}, and thus $\wt{\bP}^{i}(N_{t+s}=n_{2}|N_{t}=n_{1})$, does not depend on $t$.
We start by checking the second equality in \eqref{eq:NNotDependonY}. For any $t,s\geq 0$, $j\in\mathbf{E}$, and $0\leq n_{1}\leq n_{2}\leq n$,
\begin{align*}
\wt{\bP}^{i}\left(\left.N_{t+s}=n_{2}\,\right|N_{t}=n_{1},Y_{t}=j\right)&=\sum_{k\in\mathbf{E}}\wt{\bP}^{i}\left(\left.N_{t+s}=n_{2},Y_{t+s}=k\,\right|N_{t}=n_{1},Y_{t}=j\right)\\
&=\sum_{k\in\mathbf{E}}e^{s\,\wt{\mG}}\left((n_{1},j),(n_{2},k)\right)=\sum_{k\in\mathbf{E}}\sum_{\ell=0}^{\infty}\frac{s^{\ell}}{\ell!}\,\wt{\mG}^{\ell}\left((n_{1},j),(n_{2},k)\right)\\
&=\sum_{\ell=0}^{\infty}\frac{s^{\ell}}{\ell!}\sum_{k\in\mathbf{E}}\wt{\mG}^{\ell}\left((n_{1},j),(n_{2},k)\right)\\
&=\sum_{\ell=0}^{\infty}\frac{s^{\ell}}{\ell!}\,\wt{\mG}^{\ell}_{N}(n_{1},n_{2})=e^{s\wt{\mG}_{N}}(n_{1},n_{2}),
\end{align*}
where we used the result of Step 1  in the last two equalities. In particular, $\wt{\bP}^{i}(N_{t+s}=n_{2}|N_{t}=n_{1},Y_{t}=j)$ does not depend on the choice of $j\in\mathbf{E}$.

As far as the first equality in \eqref{eq:TildeGtoGN}, for any $t,s\geq 0$ and $0\leq n_{1}\leq n_{2}\leq n$,
\begin{align*}
\wt{\bP}^{i}\left(\left.N_{t+s}=n_{2}\,\right|N_{t}=n_{1}\right)&=\frac{\wt{\bP}^{i}\left(N_{t+s}=n_{2},N_{t}=n_{1}\right)}{\wt{\bP}^{i}\left(N_{t}=n_{1}\right)}=\frac{\sum_{\ell\in\mathbf{E}}\wt{\bP}^{i}\left(N_{t+s}=n_{2},N_{t}=n_{1},Y_{t}=j\right)}{\sum_{j\in\mathbf{E}}\wt{\bP}^{i}\left(N_{t}=n_{1},Y_{t}=j\right)}\\
&=\frac{\sum_{j\in\mathbf{E}}\wt{\bP}^{i}\left(\left.N_{t+s}=n_{2}\,\right|N_{t}=n_{1},Y_{t}=j\right)\wt{\bP}^{i}\left(N_{t}=n_{1},Y_{t}=j\right)}{\sum_{j\in\mathbf{E}}\wt{\bP}^{i}\left(N_{t}=n_{1},Y_{t}=j\right)}\\
&=\frac{\sum_{j\in\mathbf{E}}\wt{\bP}^{i}\left(N_{t}=n_{1},Y_{t}=j\right)}{\sum_{j\in\mathbf{E}}\wt{\bP}^{i}\left(N_{t}=n_{1},Y_{t}=j\right)}\,e^{s\wt{\mG}_{N}}(n_{1},n_{2})=e^{s\wt{\mG}_{N}}(n_{1},n_{2}).
\end{align*}

\noindent
\textsf{Step 3.} We are ready to complete the proof of the proposition. Towards this end we observe that, for any $m\in\bN$,  $0=t_{0}\leq t_{1}<\ldots<t_{m}$, and any $0\leq n_{1}\leq\ldots\leq n_{m}\leq n$,
\begin{align*}
\wt{\bP}^{i}(N_{t_{m}}=n_{m}\,\mid N_{t_{m-1}}  = & \, n_{m-1},\ldots,N_{t_{1}}=n_{1})=\frac{\wt{\bP}^{i}\left(N_{t_{1}}=n_{1},\ldots,N_{t_{m}}=n_{m}\right)}{\wt{\bP}^{i}\left(N_{t_{1}}=n_{1},\ldots,N_{t_{m-1}}=n_{m-1}\right)}\\
& =\frac{\sum_{j_{1},\ldots,j_{m}\in\mathbf{E}}\wt{\bP}^{i}\left(N_{t_{1}}=n_{1},Y_{t_{1}}=j_{1};\ldots;N_{t_{m}}=n_{m},Y_{t_{m}}=j_{m}\right)}{\sum_{j_{1},\ldots,j_{m-1}\in\mathbf{E}}\wt{\bP}^{i}\left(N_{t_{1}}=n_{1},Y_{t_{1}}=j_{1};\ldots;N_{t_{m-1}}=n_{m-1},Y_{t_{m}}=j_{m-1}\right)}\\
& =\frac{\sum_{j_{1},\ldots,j_{m}\in\mathbf{E}}\prod_{k=1}^{m}\wt{\bP}^{i}\left(\left.N_{t_{k}}=n_{k},Y_{t_{k}}=j_{k}\,\right|N_{t_{k-1}}=n_{k-1},Y_{t_{k-1}}=j_{k-1}\right)}{\sum_{j_{1},\ldots,j_{m-1}\in\mathbf{E}}\prod_{k=1}^{m-1}\wt{\bP}^{i}\left(\left.N_{t_{k}}=n_{k},Y_{t_{k}}=j_{k}\,\right|N_{t_{k-1}}=n_{k-1},Y_{t_{k-1}}=j_{k-1}\right)}\\
&=\sum_{j_{m}\in\mathbf{E}}\wt{\bP}^{i}\left(\left.N_{t_{m}}=n_{m},Y_{t_{m}}=j_{m}\,\right|N_{t_{m-1}}=n_{m-1},Y_{t_{m-1}}=j_{m-1}\right)\\
&=\wt{\bP}^{i}\left(\left.N_{t_{m}}=n_{m}\,\right|N_{t_{m-1}}=n_{m-1},Y_{t_{m-1}}=j_{m-1}\right)\\
&=\wt{\bP}^{i}\left(\left.N_{t_{m}}=n_{m}\,\right|N_{t_{m-1}}=n_{m-1}\right)=e^{(t_{m}-t_{m-1})\wt{\mG}_{N}}(n_{m-1},n_{m}),
\end{align*}
where we used the Markov property of $Z=(N,Y)$ under $\wt{\bP}^{i}$ in the third equality, and the result of Step 2 in the last two equalities.
The proof is complete.
\end{proof}

Let $\widetilde \bF^Y=(\widetilde {\cF}_{t}^{Y})_{t\geq 0}$ be the filtration generated by $Y$, and let $\wt{\cF}_{\infty}^{Y}=\sigma(\bigcup_{t\geq 0}\wt{\cF}_{t}^{Y})$. For each $i\in\mathbf{E}$, we will construct a probability measure $\overline{\bP}^{i}$ on $(\wt{\Omega},\wt{\cF}_{\infty}^{Y})$ such that, the law of $Y$ under $\overline{\bP}^{i}$ is the same as the law of $X$ under $\bP^{i}$. Moreover, we will establish a connection between $\overline{\bP}^{i}$ and $\wt{\bP}^{i}$. For this purpose, we first let
\begin{align*}
S_{k}:=\inf\left\{t\geq 0\,|\,N_{t}=k\right\},\quad k=1,\ldots,n.
\end{align*}
We will now derive the joint density of $N$, and $(S_{1},\ldots,S_{n})$ under $\wt{\bP}^{i}$. For that, we set
\begin{align}\label{eq:StoT}
T_{1}:=S_{1},\quad T_{k}:=S_{k}-S_{k-1},\quad k=2,\ldots,n.
\end{align}
It is shown in \cite[Section 1.1.4]{Syski1992} that $T_{k}$'s are independent and that
\begin{align*}
\wt{\bP}^{i}\left(T_{1}>t_{1},\ldots,T_{n}>t_{n} \right)=\prod_{k=1}^{n}\,e^{-q_{k}t_{k}},\quad t_{1},\ldots,t_{n}>0,
\end{align*}
which implies that the joint density of $(T_{1},\ldots,T_{n})$ is given by
\begin{align}\label{eq:densityT}
f_{T_{1},\ldots,T_{n}}(t_{1},\ldots,t_{n})=\prod_{k=1}^{n}q_{k}\,e^{-q_{k}t_{k}},\quad t_{1},\ldots,t_{n}>0.
\end{align}
Combining \eqref{eq:StoT} and \eqref{eq:densityT},  we deduce that
\begin{align*}
f_{S_{1},\ldots,S_{n}}(s_{1},\ldots,s_{n})=\prod_{k=1}^{n}q_{k}\,e^{-q_{k}(s_{k}-s_{k-1})},\quad 0=s_{0}<s_{1}<\ldots<s_{n}.
\end{align*}

\begin{theorem}\label{thm:ConstBarPi}
For any $i\in\mathbf{E}$, any $0<s_{1}<\ldots<s_{n}$, and any cylinder set $A\in\wt{\cF}_{\infty}^{Y}$ of the form
\begin{align*}
A=\left\{\left(Y_{u_{1}},\ldots,Y_{u_{m}}\right)\in B\right\},\quad 0\leq u_{1}<u_{2}<\ldots<u_{m},\quad B\subseteq\mathbf{E}^{m},\quad m\in\bN,
\end{align*}
the limit
\begin{align}\label{eq:DefBarPiCylSet}
\overline{\bP}^{i}\left(A;s_{1},\ldots,s_{n}\right):=\lim_{\Delta s_{k}\rightarrow 0,\,k=1,\ldots,n}\frac{\wt{\bP}^{i}\left(A,\,s_{k}<S_{k}\leq s_{k}+\Delta s_{k},\,k=1,\ldots,n\right)}{\wt{\bP}^{i}\left(s_{k}<S_{k}\leq s_{k}+\Delta s_{k},\,k=1,\ldots,n\right)},
\end{align}
exists, and can be extended to a probability measure $\overline{\bP}^{i}(\cdot\,;s_{1},\ldots,s_{n})$  on $(\wt{\Omega},\wt{\cF}_{\infty}^{Y})$. Moreover, for any $A\in\wt{\cF}_{\infty}^{Y}$, the function $\overline{\bP}^{i}(A;\,\ldots)$ is Borel measurable on $\{(s_{1},\ldots,s_{n})\in\bR^{n}\,|\,0<s_{1}<\ldots<s_{n}\}$, and
\begin{align}\label{eq:TildePiBarPi}
\wt{\bP}^{i}(A)=\int_{0}^{\infty}\int_{s_{1}}^{\infty}\cdots\int_{s_{n-1}}^{\infty}\overline{\bP}^{i}(A;s_{1},\ldots,s_{n})\prod_{k=1}^{n}\left(q_{k}\,e^{-q_{k}(s_{k}-s_{k-1})}\right)ds_{n}\cdots ds_{2}\,ds_{1}.
\end{align}
\end{theorem}

In the proof of the theorem we will use the following lemma.
\begin{lemma}\label{lem:ConstBarPiCylSet}
Let us fix  $i\in\mathbf{E}$, $0<s_{1}<\ldots<s_{n}$, and let $0=k_{0}<k_{1}<\ldots<k_{n+1}$ be positive integers. In addition, let $0=u_0< u_{1}<\ldots<u_{k_{1}}\leq s_{1}<u_{k_{1}+1}<\ldots<u_{k_{2}}\leq s_{2}<\ldots\leq s_{n}<u_{k_{n}+1}<\ldots<u_{k_{n+1}}$, $i_0=i$ and $i_{1},\ldots,i_{k_{n+1}}\in\mathbf{E}$. Then, for any
cylinder set $A\in\wt{\cF}_{\infty}^{Y}$ of the form
\begin{align}\label{eq:CylSet}
A=\bigcap_{j=0}^{n}\left\{Y_{u_{k_{j}+1}}=i_{k_{j}+1},\ldots,Y_{u_{k_{j+1}}}\!=i_{k_{j+1}}\right\}
\end{align}
 we have
\begin{align}
\lim_{\Delta s_{\ell}\rightarrow 0,\,\ell=1,\ldots,n}&\frac{\wt{\bP}^{i}\left(A,\,s_{\ell}<S_{\ell}\leq s_{\ell}+\Delta s_{\ell},\,\ell=1,\ldots,n\right)}{\wt{\bP}^{i}\left(s_{\ell}<S_{\ell}\leq s_{\ell}+\Delta s_{\ell},\,\ell=1,\ldots,n\right)}
=\!\prod_{\ell=0}^{n}\!\left(\!\prod_{m=k_{\ell}+1}^{k_{\ell+1}}\!\!\!\!e^{(u_{m}-u_{m-1})
\mG_{\ell}}\!(i_{m-1},i_{m})\!\right)  \nonumber \\
& \qquad\quad \cdot\sum_{j_{1},\ldots,j_{n}\in\mathbf{E}}\prod_{\ell=1}^{n}e^{(s_{\ell}-u_{k_{\ell}})
\mG_{\ell-1}}\!(i_{k_{\ell}},j_{\ell})e^{(u_{k_{\ell}+1}-s_{\ell})\mG_{\ell}}\!(j_{\ell},i_{k_{\ell}+1}).
\label{eq:BarPiSimCylSet}
\end{align}
 In particular, for any $A\in\wt{\cF}_{\infty}^{Y}$ of the form \eqref{eq:CylSet}, the above limit is Borel measurable with respect to $(s_{1},\ldots,s_{n})$ in $\Delta_n:=\{(s_{1},\ldots,s_{n})\in\bR^{n}\,|\,0<s_{1}<\ldots<s_{n}\}$.
\end{lemma}
\begin{proof}
For  $\ell=1,\ldots,n$ choose $\Delta s_{\ell}>0$ so that, $s_{\ell}+\Delta s_{\ell}\leq u_{k_{\ell}+1}$. Then,
\begin{align*}
&\wt{\bP}^{i}\left(A,\,s_{\ell}<S_{\ell}\leq s_{\ell}+\Delta s_{\ell},\,\ell=1,\ldots,n\right)\\
&\quad =\wt{\bP}^{i}\left(Y_{u_{k_{\ell}+1}}=i_{k_{\ell}+1},\ldots,Y_{u_{k_{\ell+1}}}=i_{k_{\ell+1}},\,\ell=0,\ldots,n;\,N_{s_{\ell}}=\ell-1,N_{s_{\ell}+\Delta s_{\ell}}=\ell,\,\ell=1,\ldots,n\right)\\
&\quad =\sum_{j_{1},\ldots,j_{n},\,j_{1}',\ldots,j_{n}'\in\mathbf{E}}\wt{\bP}^{i}\left(Z_{u_{k_{\ell}+1}}=(\ell,i_{k_{\ell}+1}),\,\ldots,\,Z_{u_{k_{\ell+1}}}=(\ell,i_{k_{\ell+1}}),\,\ell=0,\ldots,n;\right.\\
&\qquad\qquad\qquad\qquad\qquad\qquad\,\,\,Z_{s_{\ell}}=(\ell-1,j_{\ell}),\,Z_{s_{\ell}+\Delta s_{\ell}}=(\ell,j_{\ell}'),\,\,\ell=1,\ldots,n\Big)\\
&\quad =\sum_{j_{1},\ldots,j_{n},\,j_{1}',\ldots,j_{n}'\in\mathbf{E}}\!\left[\prod_{\ell=0}^{n}\left(\prod_{m=k_{\ell}+1}^{k_{\ell+1}}\!e^{(u_{m}-u_{m-1})\wt{\mG}}\!\left((\ell,i_{m-1}),(\ell,i_{m})\right)\right)\right]\left(\prod_{\ell=1}^{n}e^{\Delta s_{\ell}\wt{\mG}}\!\left((\ell-1,j_{\ell}),(\ell,j_{\ell}')\right)\right)\\
&\qquad\qquad\qquad\qquad\quad\cdot\left(\prod_{\ell=1}^{n}e^{(s_{\ell}-u_{k_{\ell}})\wt{\mG}}\left((\ell-1,i_{k_{\ell}}),(\ell-1,j_{\ell})\right)e^{(u_{k_{\ell}+1}-s_{\ell}-\Delta s_{\ell})\wt{\mG}}\left((\ell,j_{\ell}'),(\ell,i_{k_{\ell}+1})\right)\right).
\end{align*}
In the above summation, the first product in the brackets provides the transition probabilities of the evolutions of $Z$ between the times $u_{k_{\ell}}$ and $u_{k_{\ell+1}}$, $\ell=0,\ldots,n$, the second product gives the transition probabilities of the evolutions of $Z$  between the times $s_{\ell}$ and $s_{\ell}+\Delta s_{\ell}$, for each $\ell=1,\ldots,n$, and the third product denotes the transition probabilities of the evolutions of $Z$  between the times $u_{k_{\ell}}$ and  $s_{\ell}$, and  between the times $s_{\ell}+\Delta s_{\ell}$ and $u_{k_{\ell}+1}$, for each $\ell=1,\ldots,n$.

Next, for each $\ell=1,\ldots,n$,
\begin{align*}
\lim_{\Delta s_{\ell}\rightarrow 0}\frac{1}{\Delta s_{\ell}}\,e^{\Delta s_{\ell}\wt{\mG}}\left((\ell-1,j_{\ell}),(\ell,j_{\ell}')\right)=\wt{\mG}\left((\ell-1,j_{\ell}),(\ell,j_{\ell}')\right)=\left\{\begin{array}{ll} q_{\ell}, &\text{if }\,j_{\ell}=j_{\ell}', \\ 0, &\text{otherwise}. \end{array}\right.
\end{align*}
Hence,
\begin{align}
&\lim_{\Delta s_{\ell}\rightarrow 0,\,\ell=1,\ldots,n}\frac{1}{\Delta s_{1}\cdots\Delta s_{n}}\,\wt{\bP}^{i}\left(A,\,s_{\ell}<S_{\ell}\leq s_{\ell}+\Delta s_{\ell},\,\ell=1,\ldots,n\right)\nonumber\\
&\quad =\prod_{\ell=0}^{n}\left(\prod_{m=k_{\ell}+1}^{k_{\ell+1}}e^{(u_{m}-u_{m-1})\wt{\mG}}\!\left((\ell,i_{m-1}),(\ell,i_{m})\right)\right)\nonumber\\
\label{eq:LimProbASell} &\qquad\cdot\sum_{j_{1},\ldots,j_{n}\in\mathbf{E}}\,\prod_{\ell=1}^{n}\left(q_{\ell}\,e^{(s_{\ell}-u_{k_{\ell}})\wt{\mG}}\left((\ell-1,i_{k_{\ell}}),(\ell-1,j_{\ell})\right)e^{(u_{k_{\ell}+1}-s_{\ell})\wt{\mG}}\left((\ell,j_{\ell}),(\ell,i_{k_{\ell}+1})\right)\right).
\end{align}
Note that, for any  $j_{1},j_{2}\in\mathbf{E}$, and any $k\in\mathbb{N}$,
\begin{align*}
\wt{\mG}^{k}\left((\ell,j_{1}),(\ell,j_{2})\right)&=(\mG_{\ell}-q_{\ell+1}\mI)^{k}(j_{1},j_{2}),\quad\ell=0,\ldots,n-1,\\
\wt{\mG}^{k}\left((n,j_{1}),(n,j_{2})\right)&=\mG_{n}^{k}(j_{1},j_{2}),
\end{align*}
so that, for $t\geq 0$, we have
\begin{align*}
e^{t\,\wt{\mG}}\left((\ell,j_{1}),(\ell,j_{2})\right)&=e^{t\,(\mG_{\ell}-q_{\ell+1}\mI)}(j_{1},j_{2})=e^{-q_{\ell+1}t}\,e^{t\,\mG_{\ell}}(j_{1},j_{2}),\quad\ell=0,\ldots,n-1,\\
e^{t\,\wt{\mG}}\left((n,j_{1}),(n,j_{2})\right)&=e^{t\,\mG_{n}}(j_{1},j_{2}).
\end{align*}
This, together with \eqref{eq:LimProbASell}, implies that
\begin{align*}
&\lim_{\Delta s_{\ell}\rightarrow 0,\,\ell=1,\ldots,n}\frac{1}{\Delta s_{1}\cdots\Delta s_{n}}\,\wt{\bP}^{i}\left(A,\,s_{\ell}<S_{\ell}\leq s_{\ell}+\Delta s_{\ell},\,\ell=1,\ldots,n\right)\\
&\quad =e^{-\sum_{\ell=1}^{n}q_{\ell}(u_{k_{\ell}}-u_{k_{\ell-1}})}\cdot\prod_{\ell=0}^{n}\left(\prod_{m=k_{\ell}+1}^{k_{\ell+1}}e^{(u_{m}-u_{m-1})\mG_{\ell}}(i_{m-1},i_{m})\right)\\ &\qquad\,\,e^{-\sum_{\ell=1}^{n}q_{\ell}(s_{\ell}-u_{k_{\ell}})}e^{-\sum_{\ell=1}^{n-1}q_{\ell}(u_{k_{\ell}+1}-s_{\ell})}\!\!\!\sum_{j_{1},\ldots,j_{n}\in\mathbf{E}}\prod_{\ell=1}^{n}\!\left(q_{\ell}e^{(s_{\ell}-u_{k_{\ell}})\mG_{\ell-1}}(i_{k_{\ell}},j_{\ell})e^{(u_{k_{\ell}+1}-s_{\ell})\mG_{\ell}}(j_{\ell},i_{k_{\ell}+1})\right)\\
&\quad =e^{-\sum_{\ell=1}^{n}q_{\ell}(s_{\ell}-s_{\ell-1})}\cdot\prod_{\ell=0}^{n}\left(\prod_{m=k_{\ell}+1}^{k_{\ell+1}}e^{(u_{m}-u_{m-1})\mG_{\ell}}(i_{m-1},i_{m})\right)\\
&\qquad\,\cdot\sum_{j_{1},\ldots,j_{n}\in\mathbf{E}}\,\prod_{\ell=1}^{n}\left(q_{\ell}\,e^{(s_{\ell}-u_{k_{\ell}})\mG_{\ell-1}}(i_{k_{\ell}},j_{\ell})\,e^{(u_{k_{\ell}+1}-s_{\ell})\mG_{\ell}}(j_{\ell},i_{k_{\ell}+1})\right)\\
&\quad =\left(\prod_{\ell=1}^{n}q_{\ell}\,e^{-q_{\ell}(s_{\ell}-s_{\ell-1})}\right)\cdot\left[\prod_{\ell=0}^{n}\left(\prod_{m=k_{\ell}+1}^{k_{\ell+1}}e^{(u_{m}-u_{m-1})\mG_{\ell}}(i_{m-1},i_{m})\right)\right]\\
&\qquad\,\,\cdot\sum_{j_{1},\ldots,j_{n}\in\mathbf{E}}\,\prod_{\ell=1}^{n}\left(e^{(s_{\ell}-u_{k_{\ell}})\mG_{\ell-1}}(i_{k_{\ell}},j_{\ell})\,e^{(u_{k_{\ell}+1}-s_{\ell})\mG_{\ell}}(j_{\ell},i_{k_{\ell}+1})\right).
\end{align*}
Finally, in view of the above and the fact that
\begin{align}\label{eq:nice-fact}
\lim_{\Delta s_{\ell}\rightarrow 0,\,\ell=1,\ldots,n}\frac{1}{\Delta s_{1}\cdots\Delta s_{n}}\,\wt{\bP}^{i}\left(s_{\ell}<S_{\ell}\leq s_{\ell}+\Delta s_{\ell},\,\ell=1,\ldots,n\right)=\prod_{\ell=1}^{n}q_{\ell}\,e^{-q_{\ell}(s_{\ell}-s_{\ell-1})},
\end{align}
we obtain \eqref{eq:BarPiSimCylSet}. The proof is complete.
\end{proof}

We are now ready to prove Theorem \ref{thm:ConstBarPi}.

\begin{proof}[Proof of Theorem \ref{thm:ConstBarPi}]
Let $\cC$ be the collection of all cylinder sets in $\wt{\cF}_{\infty}^{Y}$ of the form
\begin{align*}
C=\left\{\left(Y_{u_{1}},\ldots,Y_{u_{m}}\right)\in B\right\},\quad 0\leq u_{1}<u_{2}<\ldots<u_{m},\quad B\subseteq\mathbf{E}^{m},\quad m\in\bN.
\end{align*}
Clearly, $\cC$ is an algebra.

We first show that for any $C\in \cC$ the limit in \eqref{eq:DefBarPiCylSet} exists and that an explicit formula for it can be derived. In fact,  Lemma \ref{lem:ConstBarPiCylSet} shows that the limit in \eqref{eq:DefBarPiCylSet} exists, and belongs to $[0,1]$, for all the cylinder sets of the form \eqref{eq:CylSet}. Thus, for  a  cylinder set $C\in \cC$ an explicit formula for the limit on the right-hand side of \eqref{eq:DefBarPiCylSet} can be obtained as follows. First, we refine the partition $0\leq u_{1}<u_{2}<\ldots<u_{m}$ so that each subinterval of the partition $0<s_{1}<\ldots<s_{n}$ contains at least one of the $u_i$'s. Clearly, since $B_m$ is finite, $A$ can be decomposed into a finite union of disjoint cylinder sets of the form \eqref{eq:CylSet} on the refined partition. Moreover,  \eqref{eq:BarPiSimCylSet} provides an explicit formula for the limit in \eqref{eq:DefBarPiCylSet} for each of those cylinder sets of the form \eqref{eq:CylSet} on the refined partition. Finally,  taking the finite sum over all those limits, we obtain the limit in \eqref{eq:DefBarPiCylSet} for $C$. In particular, for every cylinder set $C$, the limit in \eqref{eq:DefBarPiCylSet} is Borel measurable with respect to $(s_{1},\ldots,s_{n})$ in $\Delta_{n}$.

In the second step we will demonstrate that the limit in \eqref{eq:DefBarPiCylSet} can be extended to a probability measure on $\sigma(\cC)=\wt{\cF}_{\infty}^{Y}$.
We start from verifying the countable additivity of $\overline{\bP}^{i}(\cdot\,;s_{1},\ldots,s_{n})$ on $\cC$ for any fixed $0<s_{1}<\ldots<s_{n}$.

Since $\mathbf{E}$ is a finite set, if $(C_{k})_{k\in\bN}$ is a sequence of disjoint cylinder sets in $\cC$ such that their union also belongs to $\cC$, then only finite many of them are non-empty. Therefore, it suffices to verify the finite additivity of $\overline{\bP}^{i}(\cdot\,;s_{1},\ldots,s_{n})$ on $\cC$. Let $C_{1},\ldots,C_{k}\in\cC$ be disjoint cylinder sets, then there exists $m\in\mathbb{N}$ and $0\leq u_{1}<u_{2}<\ldots<u_{m}$, such that
\begin{align*}
C_{\ell}=\left\{\left(Y_{u_{1}},\ldots,Y_{u_{m}}\right)\in B_{\ell}\right\}\,\,\,\,\text{for some }\,B_{\ell}\subseteq\mathbf{E}^{m},\quad\ell=1,\ldots,k.
\end{align*}
Each $\overline{\bP}^{i}(C_{\ell}\,;s_{1},\ldots,s_{n})$ can be represented as
\begin{align*}
\overline{\bP}^{i}(C_{\ell}\,;s_{1},\ldots,s_{n})=\sum_{A_{\ell}\in\cC_{\ell}}\overline{\bP}^{i}(A_{\ell}\,;s_{1},\ldots,s_{n}),\quad j=1,\ldots,k,
\end{align*}
where $\cC_{\ell}$, $\ell=1,\ldots,k$, are disjoint classes of disjoint simple cylinder sets. Therefore, we have
\begin{align*}
\sum_{\ell=1}^{k}\overline{\bP}^{i}(C_{\ell}\,;s_{1},\ldots,s_{n})&=\sum_{\ell=1}^{k}\sum_{A_{\ell}\in\cC_{\ell}}\overline{\bP}^{i}(A_{\ell}\,;s_{1},\ldots,s_{n})\\
&=\sum_{A\in\cC_{1}\cup\cdots\cup\cC_{k}}\overline{\bP}^{i}(A\,;s_{1},\ldots,s_{n})=\overline{\bP}^{i}\left(\bigcup_{\ell=1}^{k}C_{\ell}\,;s_{1},\ldots,s_{n}\right).
\end{align*}
Note that for any $0<s_{1}<\ldots<s_{n}$, $\overline{\bP}^{i}(C\,;s_{1},\ldots,s_{n})\leq 1$ for all $C\in\cC$. By the Carath\'{e}odory extension theorem, for any $0<s_{1}<\ldots<s_{n}$, $\overline{\bP}^{i}(\cdot\,;s_{1},\ldots,s_{n})$ can be uniquely extended to a probability measure on $(\wt{\Omega},\wt{\cF}_{\infty}^{Y})$.

Let $\Delta_{n}:=\{(s_{1},\ldots,s_{n})\in\bR^{n}\,|\,0<s_{1}<\ldots<s_{n}\}$
and
\begin{align*}
\cD_{1}:=\left\{\left.A\in\wt{\cF}_{\infty}^{Y}\,\right|\overline{\bP}^{i}(A\,;\,\cdot,\cdots,\cdot)\,\,\text{is Borel measurable on }\Delta_{n}\right\}.
\end{align*}
We will show that $\cD_{1}=\wt{\cF}_{\infty}^{Y}$. Towards this end, we first observe that  \eqref{eq:DefBarPiCylSet} and \eqref{eq:BarPiSimCylSet} imply that, for any $A\in \cC$,  $\overline{\bP}^{i}(A\,;\,\cdot,\cdots,\cdot)$ is Borel measurable with respect to $(s_{1},\ldots,s_{n})$ on $\Delta_{n}$, and thus $\cD_{1}\supset\cC$.

Next, we will show that $\cD_{1}$ is a monotone class. For this, let $(A_{k})_{k\in\mathbb{N}}\subset\cD_{1}$ be an increasing sequence of events, so that, for any $0<s_{1}<\ldots<s_{n}$, we have
\begin{align*}
\overline{\bP}^{i}\left(\bigcup_{k=1}^{\infty}A_{k}\,;\,s_{1},\ldots,s_{n}\right)=\lim_{m\rightarrow\infty}\overline{\bP}^{i}\left(A_{m}\,;\,s_{1},\ldots,s_{n}\right).
\end{align*}
Thus, $\overline{\bP}^{i}(\cup_{k}A_{k}\,;\cdot,\cdots,\cdot)$,  being a limit of a sequence of Borel measurable functions on $\Delta_{n}$, is Borel measurable on $\Delta_{n}$, and hence $\cup_{k}A_{k}\in\cD_{1}$.
Similarly, one can show that if $(A_{k})_{k\in\mathbb{N}}\subset\cD_{1}$ is a decreasing sequence of events, then $\cap_{k}A_{k}\in\cD_{1}$. Therefore, $\cD_{1}$ is a monotone class, and by the monotone class theorem $\cD_{1}=\sigma(\cC)=\wt{\cF}_{\infty}^{Y}$.

It remains to show that \eqref{eq:TildePiBarPi} holds true.   In view of \eqref{eq:DefBarPiCylSet} and \eqref{eq:nice-fact}, for any cylinder set $A\in\cC$,
\begin{align*}
\overline{\bP}^{i}(A\,;s_{1},\ldots,s_{n})&=\lim_{\Delta s_{k}\rightarrow 0,\,k=1,\ldots,n}\frac{\wt{\bP}^{i}\left(A,\,s_{k}<S_{k}\leq s_{k}+\Delta s_{k},\,k=1,\ldots,n\right)}{\wt{\bP}^{i}\left(s_{k}<S_{k}\leq s_{k}+\Delta s_{k},\,k=1,\ldots,n\right)}\\
&=\frac{\lim_{\Delta s_{k}\rightarrow 0,\,k=1,\ldots,n}(\Delta s_{1}\cdots\Delta s_{n})^{-1}\,\wt{\bP}^{i}\left(A,\,s_{k}<S_{k}\leq s_{k}+\Delta s_{k},\,k=1,\ldots,n\right)}{\lim_{\Delta s_{k}\rightarrow 0,\,k=1,\ldots,n}(\Delta s_{1}\cdots\Delta s_{n})^{-1}\,\wt{\bP}^{i}\left(s_{k}<S_{k}\leq s_{k}+\Delta s_{k},\,k=1,\ldots,n\right)}\\
&=\frac{\partial^{n}}{\partial s_{1}\cdots\partial s_{n}}\wt{\bP}^{i}\left(A,\,S_{k}\leq s_{k},\,k=1,\ldots,n\right)\cdot\left(\prod_{k=1}^{n}q_{k}\,e^{-q_{k}(s_{k}-s_{k-1})}\right)^{-1}.
\end{align*}
Hence, for any $A\in\cC$,
\begin{align*}
&\int_{0}^{\infty}\int_{s_{1}}^{\infty}\cdots\int_{s_{n-1}}^{\infty}\overline{\bP}^{i}(A;s_{1},\ldots,s_{n})\prod_{k=1}^{n}q_{k}\,e^{-q_{k}(s_{k}-s_{k-1})}\dif s_{1}\cdots \dif s_{n}\\
&\quad=\int_{0}^{\infty}\int_{s_{1}}^{\infty}\cdots\int_{s_{n-1}}^{\infty}\frac{\partial^{n}}{\partial s_{1}\cdots\partial s_{n}}\wt{\bP}^{i}\left(A,\,S_{k}\leq s_{k},\,k=1,\ldots,n\right)\dif s_{1}\cdots \dif s_{n}=\wt{\bP}^{i}(A),
\end{align*}
and thus  $\cC \subset \cD_{2}$, where $\cD_{2}:=\left\{\left.A\in\wt{\cF}_{\infty}^{Y}\,\right|\text{\eqref{eq:TildePiBarPi} holds for }A\right\}$. Next, for any increasing sequence of events $(A_{k})_{k\in\bN}\subset\cD_{2}$, we have that
\begin{align*}
\wt{\bP}^{i}\left(\bigcup_{k=1}^{\infty}A_{k}\right)=\lim_{k\rightarrow\infty}\wt{\bP}^{i}(A_{k})&=\lim_{k\rightarrow\infty}\int_{0}^{\infty}\!\!\!\int_{s_{1}}^{\infty}\!\!\cdots\!\int_{s_{n-1}}^{\infty}\!\overline{\bP}^{i}(A_{k};s_{1},\ldots,s_{n})\prod_{\ell=1}^{n}q_{\ell}\,e^{-q_{\ell}(s_{\ell}-s_{\ell-1})}\dif s_{1}\cdots \dif s_{n}\\
&=\int_{0}^{\infty}\!\!\!\int_{s_{1}}^{\infty}\!\!\cdots\!\int_{s_{n-1}}^{\infty}\!\overline{\bP}^{i}\!\left(\bigcup_{k=1}^{\infty}A_{k};s_{1},\ldots,s_{n}\right)\!\prod_{\ell=1}^{n}q_{\ell}\,e^{-q_{\ell}(s_{\ell}-s_{\ell-1})}\dif s_{1}\cdots \dif s_{n},
\end{align*}
where the last equality follows from the dominated convergence theorem as well as the fact that $\overline{\bP}^{i}(A_{k};s_{1},\ldots,s_{n})\leq 1$, for all $k\in\mathbb{N}$ and $0<s_{1}<\ldots<s_{n}$. Hence, $\cup_{k}A_{k}\in\cD_{2}$. Similarly, one can show that if $(A_{k})_{k\in\bN}\subset\cD_{2}$ is a decreasing sequence, then $\cap_{k}A_{k}\in\cD_{2}$. Therefore, $\cD_{2}$ is a monotone class, and by the monotone class theorem $\cD_{2}=\sigma(\cC)=\wt{\cF}_{\infty}^{Y}$.
This completes the proof.
\end{proof}

Next, we will prove  that the law of $Y$ under $\overline{\bP}^{i}$ is the same as that of $X$ under $\bP^{i}$.
As usual,  $\overline{\bE}^{i}(\cdot\,;s_{1},\ldots,s_{n})$ will denote the expectation associated with $\overline{\bP}^{i}(\cdot\,;s_{1},\ldots,s_{n})$, for $i\in\mathbf{E}$ and $0<s_{1}<\ldots<s_{n}$. In the sequel, if there is no ambiguity, we will omit the parameters $s_{1},\ldots,s_{n}$ in $\overline{\bP}^{i}$ and $\overline{\bE}^{i}$.

\begin{theorem}\label{thm:YBarPiXPi}
For any $i\in\mathbf{E}$ and $0<s_{1}<\ldots<s_{n}$, under $\overline{\bP}^{i}$, $Y$ is a time-inhomogeneous Markov chain with generator $\mG=\{\mG_{t},t\geq 0\}$. In particular, $X$ and $Y$ have the same law under respective probability measures $\bP^{i}$ and $\overline{\bP}^{i}$.
\end{theorem}
\begin{proof}
Let $u_0,u_{1},\ldots,u_{m}$ be such that
\begin{align*}
0=u_{0}\leq u_{1}<\ldots<u_{k_{1}}\leq s_{1}<u_{k_{1}+1}<\ldots<u_{k_{2}}\leq s_{2}<\ldots\leq s_{n}<u_{k_{n}+1}<\ldots<u_{k_{n+1}}=u_{m}.
\end{align*}
By \eqref{eq:BarPiSimCylSet}, for any $i_{1},\ldots,i_{m}\in\mathbf{E}$,
\begin{align*}
\overline{\bP}^{i}(Y_{u_{m}}=i_{m}\mid & Y_{u_{m-1}}=i_{m-1},\ldots,Y_{u_{1}}=i_{1})=\frac{\overline{\bP}^{i}\left(Y_{u_{1}}=i_{1},\ldots,Y_{u_{m}}=i_{m}\right)}
{\overline{\bP}^{i}\left(Y_{u_{1}}=i_{1},\ldots,Y_{u_{m-1}}=i_{m-1}\right)}\\
&=\frac{\prod_{\ell=0}^{n}\left(\prod_{p=k_{\ell}+1}^{k_{\ell+1}}\!e^{(u_{p}-u_{p-1})
\mG_{\ell}}(i_{p-1},i_{p})\right)}{\left[\prod_{\ell=0}^{n-1}\left(\prod_{p=k_{\ell}+1}^{k_{\ell+1}}\!e^{(u_{p}-u_{p-1})\mG_{\ell}}(i_{p-1},i_{p})\right)\right]\left(\prod_{p=k_{n}+1}^{k_{n+1}-1}\!e^{(u_{p}-u_{p-1})\mG_{\ell}}(i_{p-1},i_{p})\right)}\\
&=e^{(u_{m}-u_{m-1})\mG_{n}}(i_{m-1},i_{m}).
\end{align*}
On the other hand, by \eqref{eq:BarPiSimCylSet} again,
\begin{align*}
\overline{\bP}^{i}  (Y_{u_{m}}&=i_{m}\mid Y_{u_{m-1}}=i_{m-1})=\frac{\overline{\bP}^{i}\left(Y_{u_{m}}=i_{m}\,Y_{u_{m-1}}=i_{m-1}\right)}{\overline{\bP}^{i}\left(Y_{u_{m-1}}=i_{m-1}\right)}\\
&=\frac{\sum_{i_{1},\ldots,i_{m-2}\in\mathbf{E}}\overline{\bP}^{i}\left(Y_{u_{1}}=i_{1},\ldots,Y_{u_{m}}=i_{m}\right)}{\sum_{i_{1},\ldots,i_{m-2}\in\mathbf{E}}\overline{\bP}^{i}\left(Y_{u_{1}}=i_{1},\ldots,Y_{u_{m-1}}=i_{m-1}\right)}\\
&=\frac{\sum_{i_{1},\ldots,i_{m-2}\in\mathbf{E}}\prod_{\ell=0}^{n}\left(\prod_{p=k_{\ell}+1}^{k_{\ell+1}}\!e^{(u_{p}-u_{p-1})\mG_{\ell}}(i_{p-1},i_{p})\right)}{\sum_{i_{1},\ldots,i_{m-2}\in\mathbf{E}}\left[\prod_{\ell=0}^{n-1}\left(\prod_{p=k_{\ell}+1}^{k_{\ell+1}}\!e^{(u_{p}-u_{p-1})\mG_{\ell}}(i_{p-1},i_{p})\right)\right]\left(\prod_{p=k_{n}+1}^{k_{n+1}-1}\!e^{(u_{p}-u_{p-1})\mG_{\ell}}(i_{p-1},i_{p})\right)}\\
&=e^{(u_{m}-u_{m-1})\mG_{n}}(i_{m-1},i_{m}).
\end{align*}
Analogous argument carries for any $u_0<u_{1}<\ldots<u_{m}$, which completes the proof.
\end{proof}

In analogy to $\varphi_t$ and $\tau_t^+$ we now define an additive functional $\psi$ given as
$\psi_{t}:=\int_{0}^{t}v(Y_{u})\dif u,\ t\geq 0$, and we consider the following first passage time $\rho_{t}^{+}:=\inf\left\{\left.r\geq 0\,\right|\psi_{r}>t\right\}, \ t\geq 0$.

We end this part of this section with the following corollary to  Theorem \ref{thm:YBarPiXPi}.
\begin{corollary}\label{thm:XYWHRelation}
For any $(s_{1},\ldots,s_{n})$ in $\Delta_{n}$, $c>0$, and $t>0$,
\begin{align}\label{eq:rho0+}
\Pi_{c}^{+}(i,j;s_{1},\ldots,s_{n})&=\overline{\bE}^{i}\left(e^{-c\rho_{0}^{+}}\1_{\{Y_{\rho_{0}^{+}}=j\}};s_{1},\ldots,s_{n}\right),\quad i\in\mathbf{E}^{-},\,j\in\mathbf{E}^{+},\\
\label{eq:rhot+} \Psi_{c}^{+}(t,i,j;s_{1},\ldots,s_{n})&=\overline{\bE}^{i}\left(e^{-c\rho_{t}^{+}}\1_{\{Y_{\rho_{t}^{+}}=j\}};s_{1},\ldots,s_{n}\right),\quad i\in\mathbf{E}^{+},\,j\in\mathbf{E}^{+}.
\end{align}
In particular, $\Pi_{c}^{+}(i,j;s_{1},\ldots,s_{n})$ and $\Psi_{c}^{+}(t,i,j;s_{1},\ldots,s_{n})$ are Borel measurable with respect to $(s_{1},\ldots,s_{n})$ in $\Delta_{n}$.
\end{corollary}


\subsection{Wiener-Hopf Factorization for $Z=(N,Y)$}

This subsection is devoted to computing the expectations on the right-hand side  in \eqref{eq:rho0+} and \eqref{eq:rhot+}. This will be done by computing the corresponding expectations related to the \textit{time-homogeneous} Markov chain $Z=(N,Y)$. The latter computation will be done using the classical Wiener-Hopf factorization results for finite state time-homogeneous Markov chains, originally derived  in \cite{Barlow1980}.

We begin with a restatement of the classical Wiener-Hopf factorization applied to $Z$. Towards this end, we let $\wt{\mathbf{E}}^{+}:=\bN_{n}\times\mathbf{E}^{+}$ and $\wt{\mathbf{E}}^{-}:=\bN_{n}\times\mathbf{E}^{-}$, and $\wt{v}:\wt{\mathbf{E}}\rightarrow\bR\setminus\{0\}$ be a function on $\wt{\mathbf{E}}$ such that $\wt{v}(k,i)=v(i)$, for all $(k,i)\in\wt{\mathbf{E}}$. Next, we define the additive functional $\wt{\varphi}$ and the corresponding first passage times as
\begin{align*}
\wt{\varphi}_{t}:=\int_{0}^{t}\wt{v}(Z_{u})\dif u
,\quad
\wt{\tau}_{t}^{\pm}:=\inf\left\{\left.r\geq 0\,\right|\pm\wt{\varphi}_{r} > t\right\},\quad t\geq 0.
\end{align*}
Let $\wt{V}:=diag\{\wt{v}(k,i):(k,i)\in\wt{\mathbf{E}}\}$ (a diagonal matrix). We denote by $\wt{\mI}^{\pm}$  the identity matrix of dimension $|\wt{\mathbf{E}}^{\pm}|$. Finally,  $\cQ(m)$ will stand for the set of $m\times m$ generator matrices (i.e., matrices with non-negative off-diagonal entries and non-positive row sums), and $\cP(m,\ell)$ will be the set of $m\times\ell$ matrices whose rows are sub-probability vectors.
\begin{theorem}\label{thm:WHZ} \cite[Theorem 1 $\&$ 2]{Barlow1980}$\,\,$
Fix $c>0$. Then,
\begin{itemize}
\item [(i)] there exists a unique quadruple of matrices $(\wt{\Lambda}_{c}^{+},\wt{\Lambda}_{c}^{-},\wt{\mG}_{c}^{+},\wt{\mG}_{c}^{-})$, where $\wt{\Lambda}_{c}^{+}\in\cP(|\wt{\mathbf{E}}^{-}|,|\wt{\mathbf{E}}^{+}|)$, $\wt{\Lambda}_{c}^{-}\in\cP(|\wt{\mathbb{E}}^{+}|,|\wt{\mathbf{E}}^{-}|)$, $\wt{\mG}_{c}^{+}\in\cQ(|\wt{\mathbf{E}}^{+}|)$, and $\wt{\mG}_{c}^{-}\in\cQ(|\wt{\mathbf{E}}^{-}|)$, such that
    \begin{align}\label{eq:WHZ}
    \wt{V}^{-1}\left(\wt{\mG}-c\,\wt{\mI}\right)\left(\begin{array}{cc} \wt{\mI}^{+} & \wt{\Lambda}_{c}^{-} \\ \wt{\Lambda}_{c}^{+} & \wt{\mI}^{-} \end{array}\right)=\left(\begin{array}{cc} \wt{\mI}^{+} & \wt{\Lambda}_{c}^{-} \\ \wt{\Lambda}_{c}^{+} & \wt{\mI}^{-} \end{array}\right)\left(\begin{array}{cc} \wt{\mG}_{c}^{+} & 0 \\ 0 & -\wt{\mG}_{c}^{-} \end{array}\right);
    \end{align}
\item [(ii)] the matrices $\wt{\Lambda}_{c}^{+}$, $\wt{\Lambda}_{c}^{-}$, $\wt{\mG}_{c}^{+}$, and $\wt{\mG}_{c}^{-}$, admit the following probabilistic representations, \begin{align}\label{eq:TildeTau0+}
    \wt{\Lambda}_{c}^{+}((k,i),(\ell,j))&=\wt{\bE}\left(\left.e^{-c\wt{\tau}_{0}^{+}}\mathbf{1}_{\{Z_{\wt{\tau}_{0}^{+}}=(\ell,j)\}}\,\right|Z_{0}=(k,i)\right),\quad (k,i)\in\wt{\mathbf{E}}^{-},\,(\ell,j)\in\wt{\mathbf{E}}^{+},\\
    \label{eq:TildeTau0-} \wt{\Lambda}_{c}^{-}((k,i),(\ell,j))&=\wt{\bE}\left(\left.e^{-c\wt{\tau}_{0}^{-}}\mathbf{1}_{\{Z_{\wt{\tau}_{0}^{-}}=(\ell,j)\}}\,\right|Z_{0}=(k,i)\right),\quad (k,i)\in\wt{\mathbf{E}}^{+},\,(\ell,j)\in\wt{\mathbf{E}}^{-},\\
    \label{eq:TildeTaut+}
    e^{t\,\wt{\mG}_{c}^{+}}((k,i),(\ell,j))&=\wt{\bE}\left(\left.e^{-c\wt{\tau}_{t}^{+}}\mathbf{1}_{\{Z_{\wt{\tau}_{t}^{+}}=(\ell,j)\}}\,\right|Z_{0}=(k,i)\right),\quad (k,i)\in\wt{\mathbf{E}}^{+},\,(\ell,j)\in\wt{\mathbf{E}}^{+},\\
    \label{eq:TildeTaut-}
    e^{t\,\wt{\mG}_{c}^{-}}((k,i),(\ell,j))&=\wt{\bE}\left(\left.e^{-c\wt{\tau}_{t}^{-}}\mathbf{1}_{\{Z_{\wt{\tau}_{t}^{-}}=(\ell,j)\}}\,\right|Z_{0}=(k,i)\right),\quad (k,i)\in\wt{\mathbf{E}}^{-},\,(\ell,j)\in\wt{\mathbf{E}}^{-},
    \end{align}
    for any $t\geq 0$.
\end{itemize}
\end{theorem}

\medskip
In what follows we will use the ``+'' part of the above formulae and only for $k=0$. Accordingly, we define (recall \eqref{eq:TildePi})
\begin{align}\label{eq:TildeTau0+0}
\wt{\Pi}_{c}^{+}(i,j,\ell)&:=\wt{\Lambda}_{c}^{+}((0,i),(\ell,j))=\wt{\bE}^{i}\left(e^{-c\wt{\tau}_{0}^{+}}\mathbf{1}_{\{Z_{\wt{\tau}_{0}^{+}}=(\ell,j)\}}\right),\quad i\in\mathbf{E}^{-},\,j\in\mathbf{E}^{+},\,\ell\in\bN,\\
\label{eq:TildeTaut+0}
\wt{\Psi}_{c}^{+}(t,i,j,\ell)&:=e^{t\,\wt{\mG}_{c}^{+}}((0,i),(\ell,j))=\wt{\bE}^{i}\left(e^{-c\wt{\tau}_{t}^{+}}\mathbf{1}_{\{Z_{\wt{\tau}_{t}^{+}}=(\ell,j)\}}\right),\quad i,j\in\mathbf{E}^{+},\,\ell\in\bN,\,t\geq 0.
\end{align}
Note that, for any $t\geq 0$, $\wt{v}(Z_{t})=v(Y_{t})$, which implies that $\wt{\varphi}_{t}=\psi_{t}$, and so $\rho_{t}^{+}=\wt{\tau}_{t}^{+}$, $\rho_{t}^{-}=\wt{\tau}_{t}^{-}$.
Hence, by taking summations over all $\ell\in\bN$ in \eqref{eq:TildeTau0+0} and \eqref{eq:TildeTaut+0}, we obtain that
\begin{align}\label{eq:SumTildeTau0+}
\wt{\bE}^{i}\left(e^{-c\rho_{0}^{+}}\1_{\{Y_{\rho_{0}^{+}}=j\}}\right)&=\sum_{\ell=0}^{n}\,\wt{\Pi}_{c}^{+}(i,j,\ell),\quad i\in\mathbf{E}^{-},\,j\in\mathbf{E}^{+},\\
\label{eq:SumTildeTaut+}
\wt{\bE}^{i}\left(e^{-c\rho_{t}^{+}}\1_{\{Y_{\rho_{t}^{+}}=j\}}\right)&=\sum_{\ell=0}^{n}\,\wt{\Psi}_{c}^{+}(t,i,j,\ell),\quad i,j\in\mathbf{E}^{+},\,t\geq 0.
\end{align}

Observe that, in view of \eqref{eq:TildePiBarPi},  if $U:\wt{\Omega}\rightarrow\bR$ is an $\wt{\cF}_{\infty}^{Y}$-measurable bounded random variable, then for any $i\in\mathbf{E}$,
\begin{align*}
\wt{\bE}^{i}(U)=\int_{0}^{\infty}\int_{s_{1}}^{\infty}\cdots\int_{s_{n-1}}^{\infty}\overline{\bE}^{i}(U;s_{1},\ldots,s_{n})\prod_{k=1}^{n}\left(q_{k}\,e^{-q_{k}(s_{k}-s_{k-1})}\right)ds_{n}\cdots ds_{2}ds_{1}.
\end{align*}
Therefore, in light of Corollary~\ref{thm:XYWHRelation}, \eqref{eq:SumTildeTau0+} and \eqref{eq:SumTildeTaut+}, we have that
\begin{align*} 
\widehat{\Pi}_{c}^{+}(i,j;q_1,\ldots,q_n)&:=\sum_{\ell=0}^{n}\,\wt{\Pi}_{c}^{+}(i,j,\ell)\\
&=\int_{0}^{\infty}\int_{s_{1}}^{\infty}\ldots\int_{s_{n-1}}^{\infty}\Pi_{c}^{+}(i,j;s_{1},\ldots,s_{n})\prod_{k=1}^{n}\left(q_{k}e^{-q_{k}(s_{k}-s_{k-1})}\right)ds_{n}\cdots ds_{2}\,ds_{1}.
\end{align*}
\begin{align*} 
\widehat{\Psi}_{c}^{+}(t,i,j;q_1,\ldots,q_n)&:=\sum_{\ell=0}^{n}\,\wt{\Psi}_{c}^{+}(t,i,j,\ell)\\
&=\int_{0}^{\infty}\!\!\!\int_{s_{1}}^{\infty}\!\!\ldots\!\int_{s_{n-1}}^{\infty}\!\Psi_{c}^{+}(t,i,j;s_{1},\ldots,s_{n})\prod_{k=1}^{n}\left(q_{k}e^{-q_{k}(s_{k}-s_{k-1})}\right)ds_{n}\!\cdots ds_{2}\,ds_{1}.
\end{align*}
By change of variables, we obtain
\begin{align*}
\widehat{\Pi}_{c}^{+}(i,j;q_1,\ldots,q_n)
&=\int_{0}^{\infty}\int_{0}^{\infty}\ldots\int_{0}^{\infty}\Pi_{c}^{+}(i,j;t_{1},\ldots,t_{1}+\ldots+t_{n})\prod_{k=1}^{n}\left(q_{k}e^{-q_{k}t_{k}}\right)dt_{1}\cdots dt_{n},\\
\widehat{\Psi}_{c}^{+}(i,j;q_1,\ldots,q_n)
&=\int_{0}^{\infty}\int_{0}^{\infty}\ldots\int_{0}^{\infty}\Psi_{c}^{+}(i,j;t_{1},\ldots,t_{1}+\ldots+t_{n})\prod_{k=1}^{n}\left(q_{k}e^{-q_{k}t_{k}}\right)dt_{1}\cdots dt_{n}.
\end{align*}
The above two equalities together with the argument in Section~\ref{sec:inverLaplaceTransform}, implies that
\begin{align*}
q_{1}^{-1}\cdots q_{n}^{-1}\widehat{\Pi}_{c}^{+}(i,j;q_1,\ldots,q_n), \qquad q_{1}^{-1}\cdots q_{n}^{-1}\widehat{\Psi}_{c}^{+}(i,j;q_1,\ldots,q_n)
\end{align*}
are well-defined for $q_{k}\in\bC^+:= \set{z\in\bC \mid \Re(z)>0},k=1,\ldots,n$, as being the Laplace transforms of $\Pi_{c}^{+}(i,j;t_{1},\ldots,t_{1}+\ldots+t_{n})$ and $\Psi_{c}^{+}(i,j;t_{1},\ldots,t_{1}+\ldots+t_{n})$, respectively.

\medskip
All the above leads to the following result, which is our main theorem, and where  we make use of the inverse multivariate Laplace transform. We refer to the Appendix for the definition and the properties of the inverse multivariate Laplace transform relevant to our set-up.

\begin{theorem}\label{th:main}  We have that

\begin{align}\label{eq:invLaplacePi+}
\Pi_{c}^{+}(i,j;s_{1},\ldots,s_{n})=\cL^{-1}\left(q_{1}^{-1}\cdots q_{n}^{-1}\widehat{\Pi}_{c}^{+}(i,j;q_1,\ldots,q_n)\right)(s_{1},s_{2}-s_{1},\ldots,s_{n}-s_{n-1}),
\end{align}
for any $i\in\mathbf{E}^{-}$, $j\in\mathbf{E}^{+}$, and

\begin{align}\label{eq:invLaplacePsi+}
\Psi_{c}^{+}(t,i,j;s_{1},\ldots,s_{n})=\cL^{-1}\left(q_{1}^{-1}\cdots q_{n}^{-1}\widehat{\Psi}_{c}^{+}(t,i,j;q_1,\ldots,q_n)\right)(s_{1},s_{2}-s_{1},\ldots,s_{n}-s_{n-1}),
\end{align}
for any $t>0$, $i,j\in\mathbf{E}^{+}$, where $\cL^{-1}$ is the inverse multivariate Laplace transform.

\end{theorem}

\begin{remark}
It needs to  be stressed that we can compute the values of $\widehat{\Pi}_{c}^{+}(i,j;q_1,\ldots,q_n)$ and $\widehat{\Psi}_{c}^{+}(t,i,j;q_1,\ldots,q_n)$ only for positive values of $q_{i}$'s. Thus, Theorem \ref{th:main} may not be directly applied to compute $\Pi_{c}^{+}(i,j;s_{1},\ldots,s_{n})$ and $\Psi_{c}^{+}(t,i,j;s_{1},\ldots,s_{n})$. However, we can approximate these functions, as explained in Section \ref{subsubsec:specialcase} by using only the values of $\widehat{\Pi}_{c}^{+}(i,j;q_1,\ldots,q_n)$ and $\widehat{\Psi}_{c}^{+}(t,i,j;q_1,\ldots,q_n)$ for positive values of $q_{i}$'s.

\end{remark}

\section{Numerical Example}\label{sec:Numerical}

In this section we will illustrate our theoretical results with a simple, but telling example. We first describe a numerical method to approximate  $\Pi_{c}^{+}$ and $\Psi_{c}^{+}$, and then we proceed with its application to a concrete example.

\subsection{Numerical Procedure to approximate $\Pi_{c}^{+}$ and $\Psi_{c}^{+}$}
   We only consider $\Pi_{c}^{+}$. The procedure to approximate $\Psi_{c}^{+}$ is analogous.

According to Theorem \ref{th:main} and Section \ref{subsubsec:specialcase}, to approximate $\Pi_{c}^{+}$, we  need  to compute
$
\widehat{\Pi}_{c}^{+}(i,j;q_1,\ldots,q_n)
$
for any $q_{1},\ldots,q_{n}>0$, and then to use  the Gaver-Stehfest algorithm. Note that $\widehat{\Pi}_{c}^{+}(i,j;q_1,\ldots,q_n)$ can be computed by solving \eqref{eq:WHZ} directly using the diagonalization method of \cite{rogers_shi_1994}. However, because of the special structure of $\wt{\mG}$, we can simplify the calculation by working on matrices of smaller dimensions. Towards this end we observe that  matrices in \eqref{eq:WHZ} can be written the block form as follows,
\begin{align}
  \wt{\mG}=
  \kbordermatrix{
    & (0,\mathbf{E}^{+}) & (1,\mathbf{E}^{+}) & \cdots &  (n,\mathbf{E}^{+}) & (0,\mathbf{E}^{-})  & (1,\mathbf{E}^{-}) & \cdots & (n,\mathbf{E}^{-}) \\
    (0,\mathbf{E}^{+}) & \mA_{1}-q_{1}\mI^{+} & q_{1}\mI^{+} & \cdots & 0 & \mB_{1} & 0 & \cdots & 0\\
    (1,\mathbf{E}^{+}) & 0 & \mA_{2}-q_{2}\mI^{+} & \cdots & 0 & 0 & \mB_{2} & \cdots & 0\\
    \vdots & \vdots & \vdots & \ddots & \vdots & \vdots & \vdots & \ddots & \vdots \\
    (n-1,\mathbf{E}^{+}) & 0 & 0 & \cdots  & q_{n}\mI^{+} & 0 & 0 & \cdots & 0\\
    (n,\mathbf{E}^{+}) & 0 & 0 & \cdots &  \mA_{n+1} & 0 & 0 & \cdots & \mB_{n+1}\\
    (0,\mathbf{E}^{-}) & \mC_{1} & 0 & \cdots & 0 & \mD_{1}-q_{1}\mI^{-} & q_{1}\mI^{-} & \cdots& 0\\
    (1,\mathbf{E}^{-}) & 0 & \mC_{2} & \cdots & 0 & 0 & \mD_{2}-q_{2}\mI^{-} & \cdots & 0\\
    \vdots & \vdots & \vdots & \ddots & \vdots & \vdots & \vdots & \ddots & \vdots \\
    (n-1,\mathbf{E}^{-}) & 0 & 0 & \cdots & 0 & 0 & 0 & \cdots & q_{n}\mI^{-}\\
    (n,\mathbf{E}^{-}) & 0 & 0 & \cdots & \mC_{n+1} & 0 & 0 & \cdots & \mD_{n+1}
  },
\end{align}
\begin{align}\label{eq:blockV}
  \wt{\mV}=
  \kbordermatrix{
    & (0,\mathbf{E}^{+}) & (1,\mathbf{E}^{+}) & \cdots & (n,\mathbf{E}^{+}) & (0,\mathbf{E}^{-})  & (1,\mathbf{E}^{-}) & \cdots & (n,\mathbf{E}^{-}) \\
    (0,\mathbf{E}^{+}) & \mV^{+} & 0 & \cdots & 0 & 0 & 0 & \cdots& 0\\
    (1,\mathbf{E}^{+}) & 0 & \mV^{+} & \cdots & 0 & 0 & 0 & \cdots& 0\\
    \vdots & \vdots & \vdots & \ddots & \vdots & \vdots & \vdots & \ddots & \vdots \\
    (n-1,\mathbf{E}^{+}) & 0 & 0 & \cdots &0 & 0 & 0 & \cdots & 0\\
    (n,\mathbf{E}^{+}) & 0 & 0 & \cdots & \mV^{+} & 0 & 0 & \cdots& 0\\
    (0,\mathbf{E}^{-}) & 0 & 0 & \cdots & 0 & \mV^{-} & 0 & \cdots & 0\\
    (1,\mathbf{E}^{-}) & 0 & 0 & \cdots& 0 & 0 & \mV^{-} & \cdots& 0\\
    \vdots & \vdots & \vdots & \ddots & \vdots & \vdots & \vdots & \ddots & \vdots \\
    (n-1,\mathbf{E}^{-}) & 0 & 0 & \cdots & 0 & 0 & 0 & \cdots & 0\\
    (n,\mathbf{E}^{-}) & 0 & 0 & \cdots& 0 & 0 & 0 & \cdots & \mV^{-}
  },
\end{align}
\begin{align}\label{eq:blockPi+}
  \wt{\Lambda}_{c}^{+}=
  \kbordermatrix{
    & (0,\mathbf{E}^{+}) & (1,\mathbf{E}^{+}) & \cdots & (n-1,\mathbf{E}^{+}) & (n,\mathbf{E}^{+}) \\
    (0,\mathbf{E}^{-}) &\wt{\Lambda}_{c,00}^{+} & \wt{\Lambda}_{c,01}^{+} & \cdots & \wt{\Lambda}_{c,0,n-1}^{+} & \wt{\Lambda}_{c,0n}^{+} \\
    (1,\mathbf{E}^{-}) & 0 & \wt{\Lambda}_{c,11}^{+} & \cdots & \wt{\Lambda}_{c,1,n-1}^{+} & \wt{\Lambda}_{c,1n}^{+} \\
    \vdots & \vdots & \vdots & \ddots & \vdots & \vdots \\
    (n-1,\mathbf{E}^{-}) & 0 & 0 & \cdots & \wt{\Lambda}_{c,n-1,n-1}^{+} & \wt{\Lambda}_{c,n-1,n}^{+} \\
    (n,\mathbf{E}^{-}) & 0 & 0 & \cdots & 0 & \wt{\Lambda}_{c,nn}^{+}
  }
\end{align}
\begin{align}\label{eq:blockPi-}
  \wt{\Lambda}_{c}^{-}=
  \kbordermatrix{
    & (0,\mathbf{E}^{-}) & (1,\mathbf{E}^{-}) & \cdots & (n-1,\mathbf{E}^{-}) & (n,\mathbf{E}^{-}) \\
    (0,\mathbf{E}^{+}) & \wt{\Lambda}_{c,00}^{-} & \wt{\Lambda}_{c,01}^{-} & \cdots & \wt{\Lambda}_{c,0,n-1}^{-} & \wt{\Lambda}_{c,0n}^{-} \\
    (1,\mathbf{E}^{+}) & 0 & \wt{\Lambda}_{c,11}^{-} & \cdots & \wt{\Lambda}_{c,1,n-1}^{-} & \wt{\Lambda}_{c,1n}^{-} \\
    \vdots & \vdots & \vdots & \ddots & \vdots & \vdots \\
    (n-1,\mathbf{E}^{+}) & 0 &0 & \cdots & \wt{\Lambda}_{c,n-1,n-1}^{-} & \wt{\Lambda}_{c,n-1,n}^{-} \\
    (n,\mathbf{E}^{+}) & 0 & 0 & \cdots & 0 & \wt{\Lambda}_{c,nn}^{-}
  },
\end{align}
\begin{align}\label{eq:blockG+}
  \wt{\mG}_{c}^{+}=
  \kbordermatrix{
    & (0,\mathbf{E}^{+}) & (1,\mathbf{E}^{+}) & \cdots & (n-1,\mathbf{E}^{+}) & (n,\mathbf{E}^{+}) \\
    (0,\mathbf{E}^{+}) & \wt{\mG}_{c,00}^{+} & \wt{\mG}_{c,01}^{+} & \cdots & \wt{\mG}_{c,0,n-1}^{+} & \wt{\mG}_{c,0n}^{+} \\
    (1,\mathbf{E}^{+}) & 0 & \wt{\mG}_{c,11}^{+} & \cdots & \wt{\mG}_{c,1,n-1}^{+} & \wt{\mG}_{c,1n}^{+} \\
    \vdots & \vdots & \vdots & \ddots & \vdots & \vdots \\
    (n-1,\mathbf{E}^{+}) & 0 & 0 & \cdots & \wt{\mG}_{c,n-1,n-1}^{+} & \wt{\mG}_{c,n-1,n}^{+} \\
    (n,\mathbf{E}^{+}) & 0 & 0 & \cdots & 0 & \wt{\mG}_{c,nn}^{+}
  },
\end{align}
and
\begin{align}\label{eq:blockG-}
  \wt{\mG}_{c}^{-}=
  \kbordermatrix{
    & (0,\mathbf{E}^{-}) & (1,\mathbf{E}^{-}) & \cdots & (n-1,\mathbf{E}^{-}) & (n,\mathbf{E}^{-}) \\
    (0,\mathbf{E}^{-}) & \wt{\mG}_{c,00}^{-} & \wt{\mG}_{c,01}^{-} & \cdots & \wt{\mG}_{c,0,n-1}^{-} & \wt{\mG}_{c,0n}^{-} \\
    (1,\mathbf{E}^{-}) & 0 & \wt{\mG}_{c,11}^{-} & \cdots & \wt{\mG}_{c,1,n-1}^{-} & \wt{\mG}_{c,1n}^{-} \\
    \vdots & \vdots & \vdots & \ddots & \vdots & \vdots \\
    (n-1,\mathbf{E}^{-}) & 0 & 0 & \cdots & \wt{\mG}_{c,n-1,n-1}^{-} & \wt{\mG}_{c,n-1,n}^{-} \\
    (n,\mathbf{E}^{-}) & 0 & 0 & \cdots &0 & \wt{\mG}_{c,nn}^{-}
  }.
\end{align}
Plugging \eqref{eq:blockV}--\eqref{eq:blockG-} into \eqref{eq:WHZ} and then comparing all the block entries on both sides, we end up with the following procedure to compute the factorization recursively.

In accordance to Theorem \ref{thm:WHZ}, for any generator matrix $\mH$ and any constant $c>0$, we denote by
\begin{align}
(\Lambda_{c}^{+}(\mH), \Lambda_{c}^{-}(\mH),\mG_{c}^{+}(\mH), \mG_{c}^{-}(\mH))
\end{align}
the unique quadruple constituting the classical Wiener-Hopf factorization (cf. \cite{Barlow1980}) corresponding to $\mH$ with killing rate $c$. In order to proceed, we let $c_{k}=q_{k}+c,\quad k\geq 1$.

\medskip
We are now ready to describe the algorithm to compute $q_{1}^{-1}\cdots q_{n}^{-1}\widehat{\Pi}_{c}^{+}(i,j;q_1,\ldots,q_n)$.
\begin{enumerate}[Step 1.]
\item \textbf{Compute the first diagonal:} for $k=1,\ldots,n+1$, compute
\begin{align}
\wt{\Lambda}_{c,k-1,k-1}^{+}=\Lambda_{c_{k}}^{+}(\mG_{k}),
\end{align}
using the diagonalization method in \cite{rogers_shi_1994}.
\item \textbf{Compute the second diagonal:} for $k=1,\ldots,n$, solve the following linear system
\begin{align}
q_{k}\mI^{+}+\mB_{k}\wt{\Lambda}_{c,k-1,k}^{+}&=\mV^{+}\wt{\mG}_{c,k-1,k}^{+},\\
[\mD_{k}-c_{k}\mI^{-}]\wt{\Lambda}_{c,k-1,k}^{+}+q_{k}\wt{\Lambda}_{c,kk}^{+}&=\mV^{-}\wt{\Lambda}_{c,k-1,k-1}^{+}\wt{\mG}_{c,k-1,k}^{+}+\mV^{-}\wt{\Lambda}_{c,k-1,k}^{+}\wt{\mG}_{c,kk}^{+},
\end{align}
for $\wt{\Lambda}_{c,k-1,k}^{+}$ and $\wt{\mG}_{c,k-1,k}^{+}$.
\item \textbf{Compute the other diagonals:} for $r=2,\ldots,n$, $k=0,\ldots,n-r$, solve the linear system
\begin{align}
\mB_{k+1}\wt{\Lambda}_{c,k,k+r}^{+}&=\mV^{+}\wt{\mG}_{c,k,k+r}^{+},\\
[\mD_{k+1}-c_{k+1}\mI^{-}]\wt{\Lambda}_{c,k,k+r}^{+}+q_{k+1}\wt{\Lambda}_{c,k+1,k+r}^{+}&=\mV^{-}\sum_{j=0}^{r}\wt{\Lambda}_{c,k,k+j}^{+}\wt{\mG}_{c,k+j,k+r}^{+},
\end{align}
for $\wt{\Lambda}_{c,k,k+r}^{+}$ and $\wt{\mG}_{c,k,k+r}^{+}$.
\item \textbf{Compute}
\begin{align*}
P^{+}(q_{1},\ldots,q_{n}):=q_{1}^{-1}\cdots q_{n}^{-1}\widehat{\Pi}_{c}^{+}(i,j;q_1,\ldots,q_n)=q_{1}^{-1}\cdots q_{n}^{-1}\sum_{\ell=0}^{n}\,\wt{\Lambda}_{c,0\ell}.
\end{align*}
for $q_1,\ldots,q_n>0$.
\item \textbf{Compute the  approximate inverse Laplace transform of $P^{+}(q_{1},\ldots,q_{n})$:} use the method discussed in Section \ref{subsubsec:specialcase}.
\end{enumerate}

\begin{remark}
If $|\mathbf{E}^{+}|=|\mathbf{E}^{-}|=1$, then the matrices in Steps 1-3 become numbers. Step 1 reduces to solving $n+1$ quadratic equations for a root in $[0,1]$. In Step 2 and 3, for each loop, the system reduces to a system of two linear equations of two unknowns. Moreover, in this case, $P^{+}$ has a closed-form representation for $q_1,\ldots,q_n>0$, and hence, for any $q_1,\ldots,q_n\in\bC^+$,  as mentioned in the previous section. This allows to use general numerical inverse Laplace transform methods, not necessary the Gaver-Stehfest formula from Section~\ref{subsubsec:specialcase}. In particular, one can use Talbot approximation formula \eqref{eq:talbotInversion} presented in Section~\ref{sec:inverLaplaceTransform}, which is more efficient than  the Gaver-Stehfest under fairly general assumptions (cf. \cite{abate2006unified}).
\end{remark}

\subsection{Application in Fluid flow problems}
The Wiener-Hopf factorization for a time-homogeneous finite Markov chain was applied in \cite{rogers1994fluid} in the context of fluid models of queues. In this section, we will apply our results to a time-inhomogeneous Markov chain fluid flow problem.

First, we briefly review the classical fluid flow problem (cf \cite{mitra1988} and \cite{rogers1994fluid} for detailed discussion). Suppose we have a large water tank with capacity $a\in (0,\infty]$. On the top of the tank, there are $I_{t}\in \mathcal{I}$ pipes open at time $t$, with each pipe pouring water into the tank at rate $r^{+}$. At the bottom of the tank, there are $O_{t}\in \mathcal{O}$ taps open at time $t$, with each tap allowing water to flow out at rate $r^{-}$. We assume that $\mathcal{I}$ and $\mathcal{O}$ are finite sets.

Then, the volume $\xi_{t}$ of water in the tank at time $t$ satisfies
\begin{align*}
\frac{\dif \xi_{t}}{\dif t}=r^{+}I_{t}-r^{-}O_{t}, \quad\text{if }0<\xi_{t}<a.
\end{align*}
Moreover, if $\xi_{t}=0$, i.e. if the tank is empty, then the outflow ceases. If $\xi_{t}=a$, i.e. if the tank is full, then water flows over the top.

Let $f$ be a real valued function on $\mathcal{I}\times \mathcal{O}$. We assume $X_{t}:=f(I_{t},O_{t}),\ t\geq 0,$ is a (finite state) time-inhomogeneous Markov chain, and we denote by $\mathbf{E}$ the state space of $X$. Let
\begin{align*}
v(x):=V(r^{+},r^{-},x),\quad x\in \mathbf{E},
\end{align*}
model the water outflow/inflow, in terms of the states of $X$, so that
\begin{align*}
v(X_{t})=V(r^{+},r^-,f(I_{t},O_{t})),\quad t\ge 0
\end{align*}
represents the water outflow/inflow at time $t$.

Let $\mathbf{E}^{+}$ be the set of states of $X$ such that the water tank has greater water inflow than outflow, and let $\mathbf{E}^{-}$ be the set of states of $X$ such that the water tank has greater water outflow than inflow. The integral
\begin{align*}
\varphi_{t}=\int_{0}^{t}v(X_{u})\dif u
\end{align*}
is not exactly the water content at time $t$, because we should take into account those periods when the tank is full or empty. However, as noted in \cite{rogers1994fluid}, understanding $\varphi_{t}$, and the corresponding $\tau_{t}^{\pm}$ and $X_{\tau_{t}^{\pm}}$ allows us to easily express the quantities of interest for $\xi_{t}$ in terms of Wiener-Hopf factorization, and to further compute these quantities once we compute the Wiener-Hopf factorization numerically.

We now assume that the tank has infinite capacity, $a=\infty$, and that it contains $\ell$ amount of water at time $t=0$. Thus,  $\tau_{\ell}^{-}$ represents the first time after $t=0$ that the tank goes empty. We will compute the quantity
\begin{align}\label{eq:computeTarget}
\Pi_{c}^{-}(i,j)=\mathbb{E}^{i}\left(e^{-c\tau_{0}^{-}}\1_{\{X_{\tau_{0}^{-}}=j\}}\right),\quad i\in\mathbf{E}^{+},\, j\in\mathbf{E}^{-}.
\end{align}
Towards this end, we further assume that the tank has either an aggregate water inflow at rate $v^{+}$ or an aggregate water outflow at rate $v^{-}$. In other words,
\begin{align*}
\mathbf{E}^{+}=\{e_+\},\quad\mathbf{E}^{-}=\{e_-\},\quad v(e_+)=v^{+},\quad\text{and}\quad v(e_-)=v^{-}.
\end{align*}
Moreover, we assume that the time-inhomogeneous Markov chain $X$ has the generator
\begin{align*}
\mG_{t}=\left\{\begin{array}{lll}
& \mG_{1},\quad & s_{0}\le t <s_{1},\\
& \mG_{2},\quad &s_{1}\le t <s_{2},\\
& \mG_{3},\quad &t\ge s_{2},
\end{array}\right.
\end{align*}
where $0<s_{1}<s_{2}$.

We take the following inputs: $c=0.5, v(e_+)=2, v(e_-)=-3,s_{1}=2,s_{2}=8$,
\begin{align*}
  \mG_{0}=
  \kbordermatrix{
    & e_+ & e_- \\
    e_+ &-2 & 2\\
   e_- &1 & -1
  }, \quad
    \mG_{1}=
    \kbordermatrix{
      & e_+ & e_- \\
      e_+ &-3 & 3\\
     e_- &2 & -2
    },\quad
      \mG_{2}=
      \kbordermatrix{
        & e_+ & e_- \\
        e_+ &-5 & 5\\
       e_- &3 & -3
      }.
\end{align*}
The following table compares our result and execution time with Monte-Carlo simulation (10000 paths).\\
\begin{center}
\begin{tabular}{llr}
\hline
\multicolumn{3}{c}{Numerical Results} \\
\cline{1-3}
Method    & Wiener-Hopf & Monte-Carlo \\
\hline
$\Pi_{c}^{-}(e_+,e_-)$      & $0.6501$    & $0.6462$      \\
Execution time &  $0.15\,s$    & $3.12\,s$      \\
\hline
\end{tabular}
\end{center}
\begin{remark}
One can also compute $\Pi_{c}^{+}(e_-,e_+)$, if it is the quantity of interest in the model. Note that if we change the labels of the states from $\{e_+,e_-\}$ to $\{e_-,e_+\}$ and modify the inputs accordingly, we can compute $\Pi_{c}^{+}(e_-,e_+)$ using the same algorithm that computes $\Pi_{c}^{-}(e_+,e_-)$.
\end{remark}

\section{Appendix: Approximation of Multivariate Inverse Laplace Transform}\label{sec:inverLaplaceTransform}
For the convenience of the reader, we will briefly recall the basics of Laplace transform and its inverse. Then, we will proceed with an important result regarding the approximation of the multivariate inverse Laplace transform.

Let $f:[0,\infty)^{n}\rightarrow [0,\infty)$ be a Borel-measurable function such that
\begin{align*}
\int_{0}^{\infty}\cdots\int_{0}^{\infty}f(t_{1},\ldots,t_{n})\prod_{k=1}^{n}e^{-q_{k}t_{k}}\dif t_{1}\cdots\dif t_{n}
\end{align*}
exists for any $q_{1},\ldots,q_{n}>0$. Then, the multivariate  Laplace  transform $\widehat{f}$ of $f$, defined by
\begin{align*}
\wh{f}(q_{1},\ldots,q_{n})=\cL(f)(q_{1},\ldots,q_{n}):=\int_{0}^{\infty}\cdots\int_{0}^{\infty}f(t_{1},\ldots,t_{n})\prod_{k=1}^{n}e^{-q_{k}t_{k}}\dif t_{1}\cdots\dif t_{n},
\end{align*}
is well-defined for any $q_{k} \in \bC^+$, $k=1,\ldots,n$, where\footnote{We will denote by $\Re(z)$ the real part of $z\in\bC$, and $\mathrm{i}=\sqrt{-1}$ will be used to denote the imaginary unit.} $\bC^+:= \set{z\in\bC \mid \Re(z)>0}$ with $\Re(z)$ denoting the real part of $z\in\bC$.
The inverse multivariate Laplace transform of function $g:(\bC^+)^n \to\bC$, is the function $\check{g}$,  such that $\cL(\check{g}) = g$. We will also write
$\check{g} = \cL^{-1}(g)$. The existence and uniqueness of the inverse Laplace transform is a well understood subject (cf. \cite{Widder1941}). Although there are explicit formulas of the inverse Laplace transform for many functions, generally speaking, in many practical situations the inverse Laplace transform of a function is computed by numerical approximation technics. We refer the reader to \cite{abate2006unified}, and the references therein, for  a unified framework for numerically inverting the Laplace transform. For sake of completeness, we present here one such method - the Talbot inversion formula - for one and two dimensional case; the multidimensional case is done by analogy.

Assume that $\wh{f}$ is the Laplace transform of a function $f:(0,+\infty)\to\bC$. The Talbot inversion formula to approximate $f$ is given by
\begin{align}\label{eq:talbotInversion}
f^{b}_M(t)=\frac2{5t}\sum_{k=0}^{M-1}\Re\left(\gamma_{k}\wh{f}(\frac{\delta_{k}}{t})\right),
\end{align}
where
\begin{align}
\delta_{0}&=\frac{2M}{5},\quad\delta_{k}=\frac{22k\pi}{5}(\cot(\frac{k\pi}{M})+\mathrm{i}),\quad 0<k<M,\nonumber\\
\gamma_{0}&=\frac12 e^{\delta_{0}},\quad\gamma_{k}=\left(1+\mathrm{i}\frac{k\pi}{M}(1+\cot^{2}(\frac{k\pi}{M}))- \mathrm{i}\cot(\frac{k\pi}{M}) \right)e^{\delta_{k}},\quad 0<k<M.\label{eq:talbotGamma}
\end{align}
Analogously, given a Laplace transform $\wh{g}$ of a complex-valued function $g$ of two non-negative real variables,
the Talbot inversion formula to compute $g(t_{1},t_{2})$ numerically is given by
\begin{align*}
g^{b}_M(t_{1},t_{2})=\frac{2}{25t_{1}t_{2}}\sum_{k_{1}=0}^{M-1}\Re\left\{\gamma_{k_{1}}\sum_{k_{2}=0}^{M-1}\left[\gamma_{k_{2}}\wh{g}\left(\frac{\delta_{k_{1}}}{t_{1}},\frac{\delta_{k_{2}}}{t_{2}} \right) +\bar{\gamma}_{k_{2}}\wh{g}\left(\frac{\delta_{k_{1}}}{t_{1}},\frac{\bar{\delta}_{k_{2}}}{t_{2}} \right)\right] \right\},
\end{align*}
where $\delta_{k},\gamma_{k},0\le k<M,$ are given in \eqref{eq:talbotGamma}.

\subsection{A Special Case of Numerical Inverse Laplace Transform}\label{subsubsec:specialcase}
 Let us consider a function $f:[0,\infty)\rightarrow [0,\infty)$ and its Laplace transform $\wh{f}(q)$, for $q\in\bC^+$. It turns out that the inverse Laplace transform of $f$ can be approximated numerically by using only values of the function $\wh{f}$ on the positive real line. One such approximation is the Gaver-Stehfest formula
\begin{align}\label{eq:GaverStehfest}
f_{n}(t)=\frac{n\log 2}{t}\binom{2n}{n}\sum_{k=0}^{n}(-1)^{k}\binom{n}{k}\wh{f}\left(\frac{(n+k)\log 2}{t} \right).
\end{align}
For other methods and the comparison of their speeds of convergence we refer to \cite{abate2006unified}. Consecutive application of \eqref{eq:GaverStehfest} leads to the multivariate Gaver-Stehfest formula.

\bibliographystyle{alpha}

\begin{thebibliography}{LMRW82}



\bibitem[APU03]{Avram2003}
F.~Avram, M.~R. Pistorius, and M.~Usabel.
\newblock The two barriers ruin problem via a {W}iener {H}opf decomposition
  approach.
\newblock {\em An. Univ. Craiova Ser. Mat. Inform.}, 30(1):38--44, 2003.

\bibitem[Asm95]{Asmussen1995}
S.~Asmussen.
\newblock Stationary distributions for fluid flow models with or without
  {B}rownian noise.
\newblock {\em Comm. Statist. Stochastic Models}, 11(1):21--49, 1995.

\bibitem[AW06]{abate2006unified}
J.~Abate and W.~Whitt.
\newblock A unified framework for numerically inverting laplace transforms.
\newblock {\em INFORMS Journal on Computing}, 18(4):408--421, 2006.

\bibitem[BRW80]{Barlow1980}
M.~T. Barlow, L.~C.~G. Rogers, and D.~Williams.
\newblock Wiener-hopf factorization for matrices.
\newblock {\em Séminaire de probabilités de Strasbourg}, 14:324--331, 1980.

\bibitem[Hie14]{Hieber2014}
P.~Hieber.
\newblock First-passage times of regime switching models.
\newblock {\em Statist. Probab. Lett.}, 92:148--157, 2014.

\bibitem[HSZ16]{Hainaut2016}
D.~Hainaut, Y.~Shen, and Y.~Zeng.
\newblock How do capital structure and economic regime affect fair prices of
  bank's equity and liabilities?
\newblock {\em Annals of Operations Research}, Apr 2016.

\bibitem[JP08]{Jiang2008}
Z.~Jiang and M.~R. Pistorius.
\newblock On perpetual {A}merican put valuation and first-passage in a
  regime-switching model with jumps.
\newblock {\em Finance Stoch.}, 12(3):331--355, 2008.

\bibitem[JP12]{Jiang2012}
Z.~Jiang and M.~R. Pistorius.
\newblock Optimal dividend distribution under {M}arkov regime switching.
\newblock {\em Finance Stoch.}, 16(3):449--476, 2012.

\bibitem[JR06]{Jobert2006}
A.~Jobert and L.~C.~G. Rogers.
\newblock Option pricing with {M}arkov-modulated dynamics.
\newblock {\em SIAM J. Control Optim.}, 44(6):2063--2078, 2006.

\bibitem[KW90]{Kennedy1990}
J.~Kennedy and D.~Williams.
\newblock Probabilistic factorization of a quadratic matrix polynomial.
\newblock {\em Math. Proc. Cambridge Philos. Soc.}, 107(3):591--600, 1990.

\bibitem[LMRW82]{London1982}
R.~R. London, H.~P. McKean, L.~C.~G. Rogers, and D.~Williams.
\newblock A martingale approach to some {W}iener-{H}opf problems. {I}, {II}.
\newblock In {\em Seminar on {P}robability, {XVI}}, volume 920 of {\em Lecture
  Notes in Math.}, pages 41--67, 68--90. Springer, Berlin-New York, 1982.

\bibitem[Mit88]{mitra1988}
D.~Mitra.
\newblock Stochastic theory of a fluid model of producers and consumers coupled
  by a buffer.
\newblock {\em Adv. in Appl. Probab.}, 20(3):646--676, 1988.

\bibitem[MP11]{Mijatovic2011}
A.~Mijatovi\'c and M.~R. Pistorius.
\newblock Exotic derivatives under stochastic volatility models with jumps.
\newblock In {\em Advanced mathematical methods for finance}, pages 455--508.
  Springer, Heidelberg, 2011.

\bibitem[Rog94]{rogers1994fluid}
L.~C.~G. Rogers.
\newblock Fluid models in queueing theory and {W}iener-{H}opf factorization of
  {M}arkov chains.
\newblock {\em Ann. Appl. Probab.}, 4(2):390--413, 1994.

\bibitem[RS94]{rogers_shi_1994}
L.~C.~G. Rogers and Z.~Shi.
\newblock Computing the invariant law of a fluid model.
\newblock {\em J. Appl. Probab.}, 31(4):885--896, 1994.

\bibitem[Sys92]{Syski1992}
R.~Syski.
\newblock {\em Passage times for {M}arkov chains}, volume~1 of {\em Studies in
  Probability, Optimization and Statistics}.
\newblock IOS Press, Amsterdam, 1992.

\bibitem[Wid41]{Widder1941}
D.~V. Widder.
\newblock {\em The {L}aplace {T}ransform}.
\newblock Princeton Mathematical Series, v. 6. Princeton University Press,
  Princeton, N. J., 1941.

\bibitem[Wil91]{Williams1991}
D.~Williams.
\newblock Some aspects of {W}iener-{H}opf factorization.
\newblock {\em Philos. Trans. Roy. Soc. London Ser. A}, 335(1639):593--608,
  1991.

\bibitem[Wil08]{Williams2008}
D.~Williams.
\newblock A new look at `{M}arkovian' {W}iener-{H}opf theory.
\newblock In {\em S\'eminaire de probabilit\'es {XLI}}, volume 1934 of {\em
  Lecture Notes in Math.}, pages 349--369. Springer, Berlin, 2008.

\end{thebibliography}
{\small

\def\cprime{$'$}

}

\end{document}